\documentclass[12pt]{amsart}
\usepackage{amsmath,amsfonts,amssymb,amsthm}
\usepackage{graphicx,color}
\DeclareMathRadical{\sqrtsign}{symbols}{"70}{largesymbols}{"70}
\newlength{\figboxwidth}             
\newcommand{\grad}{\nabla}

\def\@ifundefined#1#2#3%
  {\expandafter\ifx\csname#1\endcsname\relax#2\else#3\fi}

\@ifundefined{theoremstyle}{
}{
\theoremstyle{plain} 
}
\newtheorem{theorem}{Theorem}[section]

\newtheorem{proposition}[theorem]{Proposition}
\newtheorem{lemma}[theorem]{Lemma}

\newtheorem{corollary}[theorem]{Corollary}

\@ifundefined{theoremstyle}{
}{
\theoremstyle{definition} 
}
\newtheorem{definition}[theorem]{Definition}

\newtheorem{remark}[theorem]{Remark}

\mathchardef\GG="321D
\newcommand{\mcc}[1]{{}}

\numberwithin{equation}{section}

\title[Hausdorff dimension across generic dynamical spectra]
{Continuity of Hausdorff dimension across generic dynamical Lagrange and Markov spectra}

\dedicatory{To Jean-Christophe Yoccoz (in memoriam)}

\author{Aline Cerqueira}
\thanks{A.C. was partially supported by CNPq-Brazil. Also, she thanks the hospitality of Coll\`ege de France and IMPA-Brazil during the preparation of this article.}
\address{Aline Cerqueira: 
IMPA, Instituto de Matem\'atica Pura e Aplicada,   
Estrada Dona Castorina, Jardim Bot\^anico, 
Rio de Janeiro, RJ, CEP 22460-320, Brazil.
}
\email{alineagc@gmail.com}

\author{Carlos Matheus}
\thanks{C.M. was temporarily affiliated to the UMI CNRS-IMPA (UMI 2924) during the final stages of preparation of this work and he is grateful to IMPA-Brazil for the hospitality during this period.}
\address{Carlos Matheus: Universit\'e Paris 13, Sorbonne Paris Cit\'e, LAGA, CNRS (UMR 7539), F-93439, Villetaneuse, France}
\email{matheus.cmss@gmail.com}

\author{Carlos Gustavo Moreira}
\thanks{C.G.M. was partially supported by CNPq-Brazil.}
\address{Carlos G. Moreira: 
IMPA, Instituto de Matem\'atica Pura e Aplicada,   
Estrada Dona Castorina, Jardim Bot\^anico, 
Rio de Janeiro, RJ, CEP 22460-320, Brazil.
}
\email{gugu@impa.br}

\keywords{Hausdorff dimension, horseshoes, Lagrange spectrum, Markov spectrum, surface diffeomorphisms}

\date{\today}

\begin{document}

\begin{abstract}
Let $\varphi_0$ be a smooth area-preserving diffeomorphism of a compact surface $M$ and let $\Lambda_0$ be a horseshoe of $\varphi_0$ with Hausdorff dimension strictly smaller than one. Given a smooth function $f:M\to \mathbb{R}$ and a small smooth area-preserving perturtabion $\varphi$ of $\varphi_0$, let $L_{\varphi, f}$, resp. $M_{\varphi, f}$ be the Lagrange, resp. Markov spectrum of asymptotic highest, resp. highest values of $f$ along the $\varphi$-orbits of points in the horseshoe $\Lambda$ obtained by hyperbolic continuation of $\Lambda_0$. 

We show that, for generic choices of $\varphi$ and $f$, the Hausdorff dimension of the sets $L_{\varphi, f}\cap (-\infty, t)$ vary continuously with $t\in\mathbb{R}$ and, moreover, $M_{\varphi, f}\cap (-\infty, t)$ has the same Hausdorff dimension of $L_{\varphi, f}\cap (-\infty, t)$ for all $t\in\mathbb{R}$. 
\end{abstract}

\maketitle



\section{Introduction}\label{s.introduction}

\subsection{Classical Markov and Lagrange spectra}\label{ss.classical-Markov-Lagrange} The origin of the classical Lagrange and Markov spectra lies in Number Theory. 

More precisely, given an irrational number $\alpha>0$, let 
$$k(\alpha):=\limsup_{\substack{p, q\to\infty \\ p, q\in\mathbb{N}}}(q|q\alpha-p|)^{-1}$$
be its best constant of Diophantine approximation. The set 
$$L:=\{k(\alpha):\alpha\in\mathbb{R}-\mathbb{Q}, \alpha>0, k(\alpha)<\infty\}$$ 
consisting of all finite best constants of Diophantine approximations is the so-called \emph{Lagrange spectrum}. 

Similarly, the \emph{Markov spectrum} 
$$M:=\left\{\left(\inf\limits_{(x,y)\in\mathbb{Z}^2-\{(0,0)\}} |q(x,y)|\right)^{-1}: q(x,y)=ax^2+bxy+cy^2, b^2-4ac=1\right\}$$ 
consists of the reciprocal of the minimal values over non-trivial integer vectors $(x,y)\in\mathbb{Z}^2-\{(0,0)\}$ of indefinite binary quadratic forms $q(x,y)$ with unit discriminant. 

These spectra are closely related: it is know that the Lagrange and Markov spectra are closed subsets of $\mathbb{R}$ such that $L\subset M$. The reader can find more informations about the structure of these sets on the classical book \cite{CF} of Cusick and Flahive, but let us just mention that: 
\begin{itemize}
\item Hurwitz theorem says that $\sqrt{5}=\min L$;
\item Markov showed that $L\cap(-\infty, 3)=\{k_1, k_2, \dots\}$ where $k_n$ is an explicit\footnote{E.g., $k_1=\sqrt{5}$, $k_2=2\sqrt{2}$, $k_3=\frac{\sqrt{221}}{5}$, etc.} increasing sequence of quadratic surds converging to $3$; 
\item Hall showed that $L$ contains a half-line and Freiman determined the biggest half-line contained in $L$ (namely, $[c,+\infty)$ where $c=\frac{2221564096+283748\sqrt{462}}{491993569}\simeq 4.52782956\dots$); 
\item Moreira \cite{Mo1} proved that the Hausdorff dimension of $L\cap (-\infty, t)$ varies continuously with $t\in\mathbb{R}$, and, moreover, the sets $L\cap (-\infty, t)$ and $M\cap(-\infty, t)$ have the same Hausdorff dimension for all $t\in\mathbb{R}$.
\end{itemize}

For our purposes, it is worth to point out that the Lagrange and Markov spectra have the following \emph{dynamical} interpretation\footnote{This dynamical interpretation also has a version in terms of the cusp excursions of the geodesic flow on the modular surface $\mathbb{H}/SL(2,\mathbb{Z})$, see \cite{HP} for instance.} in terms of the continued fraction algorithm. 

Denote by $[a_0,a_1,\dots]$ the continued fraction $a_0+\frac{1}{a_1+\frac{1}{\ddots}}$. Let $\Sigma=\mathbb{N}^{\mathbb{Z}}$ the space of bi-infinite sequences of positive integers, $\sigma:\Sigma\to\Sigma$ be the shift dynamics $\sigma((a_n)_{n\in\mathbb{Z}}) = (a_{n+1})_{n\in\mathbb{Z}}$, and let $f:\Sigma\to\mathbb{R}$ be the function
$$f((a_n)_{n\in\mathbb{Z}}) = [a_0, a_1,\dots] + [0, a_{-1}, a_{-2},\dots]$$
Then, 
$$L=\left\{\limsup_{n\to\infty}f(\sigma^n(\underline{\theta})):\underline{\theta}\in\Sigma\right\} \quad \textrm{and} \quad M= \left\{\sup_{n\to\infty}f(\sigma^n(\underline{\theta})):\underline{\theta}\in\Sigma\right\}$$

In the sequel, we consider the natural generalization of this dynamical version of the classical Lagrange and Markov spectra in the context of horseshoes\footnote{I.e., a non-empty compact invariant hyperbolic set of saddle type which is transitive, locally maximal, and not reduced to a periodic orbit (cf. \cite{PT} for more details).} of smooth diffeomorphisms of compact surfaces. In this setting, our main result (cf. Theorem \ref{t.A} below) will be a dynamical analog of the results of \cite{Mo1} (quoted above) on the continuity of Hausdorff dimension across Lagrange and Markov spectra.  

\subsection{Dynamical Markov and Lagrange spectra}\label{ss.dynamical-Markov-Lagrange}

Let $M$ be a surface and consider $\varphi:M\to M$ a $C^2$-diffeomorphism possessing a horseshoe $\Lambda$. 

Given $f:M\to\mathbb{R}$ a $C^r$-function, $r\geq 2$, and $t\in\mathbb{R}$, we define the \emph{dynamical Markov, resp. Lagrange, spectrum} $M_{\varphi, f}$, resp. $L_{\varphi, f}$ as 
$$M_{\varphi, f}=\{m_{\varphi, f}(x): x\in\Lambda\}, \quad \textrm{resp.} \quad L_{\varphi, f}=\{\ell_{\varphi, f}(x): x\in\Lambda\}$$
where  
$$m_{\varphi, f}(x):=\sup\limits_{n\in\mathbb{Z}} f(\varphi^n(x)), \quad \textrm{resp.} \quad \ell_{\varphi, f}(x)=\limsup\limits_{n\to+\infty} f(\varphi^n(x))$$

\begin{remark}\label{r.L<M} An elementary compactness argument (cf. Remark in Section 3 of \cite{MoRo}) shows that  
$$\{\ell_{\varphi, f}(x):x\in A\}\subset\{m_{\varphi, f}(x): x\in A\}\subset f(A)$$
whenever $A\subset M$ is a compact $\varphi$-invariant subset. 
\end{remark}

In this paper, we will be interested in the fractal geometry (Hausdorff dimension) of the sets $M_{\varphi, f}\cap (-\infty, t)$ and $L_{\varphi, f}\cap (-\infty, t)$ as $t\in\mathbb{R}$ varies. 

For this reason, we will also study the fractal geometry of 
$$\Lambda_t:=\bigcap\limits_{n\in\mathbb{Z}}\varphi^{-n}(\{y\in\Lambda: f(y)\leq t\}) = \{x\in\Lambda: m_{\varphi, f}(x)=\sup\limits_{n\in\mathbb{Z}}f(\varphi^n(x))\leq t\}$$ 
for $t\in\mathbb{R}$. 

More precisely, we will consider the following setting (and we refer to Palis-Takens book \cite{PT} for more details). Let us fix a geometrical Markov partition $\{R_a\}_{a\in\mathcal{A}}$ with sufficiently small diameter consisting of rectangles $R_a\simeq I_a^s\times I_a^u$ delimited by compact pieces $I_a^s$, resp. $I_a^u$, of stable, resp. unstable, manifolds of certain points of $\Lambda$. We define the subset $\mathcal{T}\subset\mathcal{A}^2$ of admissible transitions as the subset of pairs $(a_0, a_1)\in\mathcal{A}^2$ such that $\varphi(R_{a_0})\cap R_{a_1}\neq\emptyset$. In this way, the dynamics of $\varphi$ on $\Lambda$ is topologically conjugated to a Markov shift $\Sigma_{\mathcal{T}}\subset\mathcal{A}^{\mathbb{Z}}$ of finite type associated to $\mathcal{T}$. 

\begin{figure}[htb!]
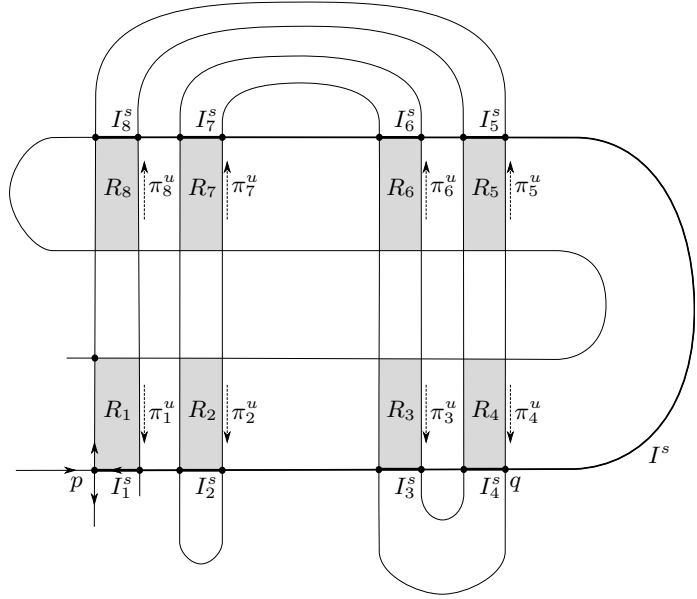
\caption{Geometry of the horseshoe $\Lambda$.}
\end{figure}

Next, we recall that the stable and unstable manifolds of $\Lambda$ can be extended to locally invariant $C^{1+\varepsilon}$-foliations in a neighborhood of $\Lambda$ for some $\varepsilon>0$. Therefore, we can use these foliations to define projections $\pi_a^u:R_a\to I_a^s\times\{i_a^u\}$ and $\pi_a^s: R_a\to \{i_a^s\}\times I_a^u$ of the rectangles into the connected components $I_a^s\times\{i_a^u\}$ and $\{i_a^s\}\times I_a^u$ of the stable and unstable boundaries of $R_a$ where $i_a^u\in\partial I_a^u$ and $i_a^s\in\partial I_a^s$ are fixed arbitrarily. Using these projections, we have the stable and unstable Cantor sets 
$$K^s=\bigcup\limits_{a\in\mathcal{A}}\pi_a^u(\Lambda\cap R_a) \quad \textrm{ and } \quad K^u=\bigcup\limits_{a\in\mathcal{A}}\pi_a^s(\Lambda\cap R_a)$$
associated to $\Lambda$. 

The stable and unstable Cantor sets $K^s$ and $K^u$ are $C^{1+\varepsilon}$-dynamically defined / $C^{1+\varepsilon}$-regular Cantor sets, i.e., the $C^{1+\varepsilon}$-maps 
$$g_s(\pi_{a_1}^u(y))=\pi_{a_0}^u(\varphi^{-1}(y))$$
for $y\in R_{a_1}\cap \varphi(R_{a_0})$
and
$$g_u(\pi_{a_0}^s(z))=\pi_{a_1}^s(\varphi(z))$$
for $z\in R_{a_0}\cap \varphi^{-1}(R_{a_1})$ are expanding of type $\Sigma_\mathcal{T}$ defining $K^s$ and $K^u$ in the sense that 
\begin{itemize}
\item the domains of $g_s$ and $g_u$ are disjoint unions $\bigsqcup\limits_{(a_0,a_1)\in\mathcal{T}} I^s(a_1, a_0)$ and $\bigsqcup\limits_{(a_0, a_1)\in\mathcal{T}} I^u(a_0, a_1)$ where $I^s(a_1, a_0)$, resp. $I^u(a_0, a_1)$, are compact subintervals of $I_{a_1}^s$, resp. $I_{a_0}^u$; 
\item for each $(a_0, a_1)\in\mathcal{T}$, the restrictions $g_s|_{I^s(a_0,a_1)}$ and $g_u|_{I^u(a_0,a_1)}$ are $C^{1+\varepsilon}$ diffeomorphisms onto $I^s_{a_0}$ and $I_{a_0}^u$ with $|Dg_s(t)|>1$, resp. $|Dg_u(t)|>1$, for all $t\in I^s(a_0, a_1)$, resp. $I^u(a_0, a_1)$ (for appropriate choices of the parametrization of $I_a^s$ and $I_a^u$); 
\item $K^s$, resp. $K^u$, are the maximal invariant sets associated to $g_s$, resp. $g_u$, that is, 
$$K^s=\bigcap\limits_{n\in\mathbb{N}}g_s^{-n}\left(\bigcup\limits_{(a_0,a_1)\in\mathcal{T}} I^s(a_1,a_0)\right) \textrm{ and } K^u=\bigcap\limits_{n\in\mathbb{N}}g_u^{-n}\left(\bigcup\limits_{(a_0,a_1)\in\mathcal{T}} I^u(a_0,a_1)\right)$$
\end{itemize} 

Moreover, we will think the intervals $I^u_a$, resp. $I^s_a$, $a\in\mathcal{A}$ inside an abstract line so that it makes sense to say that the interval $I^u_a$, resp. $I^s_a$, is located to the left or to the right of the interval $I^u_b$, resp. $I^s_b$, for $a,b\in\mathcal{A}$ (see \cite{PT}). 

The stable and unstable Cantor sets $K^s$ and $K^u$ are closely related to the geometry of the horseshoe $\Lambda$: for instance, it is well-known that 
$$\textrm{dim}(\Lambda) = \textrm{dim}(K^s)+ \textrm{dim}(K^u):=d_s+d_u$$
where $\textrm{dim}$ stands for the Hausdorff dimension, and, furthermore, $d_s=d_u$ when $\varphi$ is conservative, i.e., $\varphi$ preserves a smooth area form $\omega$. 

Partly motivated by this fact, we will study the subsets $\Lambda_t$ introduced above through its projections  
$$K_t^s=\bigcup_{a\in\mathcal{A}}\pi_a^u(\Lambda_t\cap R_a) \textrm{ and } K_t^u = \bigcup_{a\in\mathcal{A}}\pi_a^s(\Lambda_t\cap R_a)$$
on the stable and unstable Cantor sets of $\Lambda$.

\subsection{Statement of the main result} Using the notations of the previous Subsection, our main result is the following. Let $\varphi_0$ be a smooth conservative diffeomorphism of a surface $M$ possessing a horseshoe $\Lambda_0$ with Hausdorff dimension $\textrm{dim}(\Lambda_0)<1$. Denote by $\mathcal{U}$ a small $C^{\infty}$ neighborhood of $\varphi_0$ in the space $\textrm{Diff}_{\omega}^{\infty}(M)$ of smooth conservative diffeomorphisms of $M$ such that $\Lambda_0$ admits a continuation $\Lambda$ for every $\varphi\in\mathcal{U}$.

\begin{theorem}\label{t.A} If $\mathcal{U}\subset\textrm{Diff}_{\omega}^{\infty}(M)$ is sufficiently small, then there exists a Baire residual subset $\mathcal{U}^{**}\subset \mathcal{U}$ with the following property. For every $\varphi\in\mathcal{U}^{**}$, there exists a $C^r$-open and dense subset $\mathcal{R}_{\varphi,\Lambda}\subset C^r(M,\mathbb{R})$ such that the functions 
$$t\mapsto d_s(t):=\textrm{dim}(K_t^s) \quad \textrm{ and } \quad t\mapsto d_u(t):=\textrm{dim}(K_t^u)$$
are continuous and 
$$d_s(t)+d_u(t) = 2 d_u(t) = \textrm{dim}(L_{\varphi, f}\cap (-\infty, t)) = \textrm{dim}(M_{\varphi, f}\cap (-\infty, t))$$
whenever $f\in \mathcal{R}_{\varphi,\Lambda}$.
\end{theorem}

\begin{remark}\label{r.tA} Our proof of Theorem \ref{t.A} shows that $d_s(t)$, resp. $d_u(t)$, coincide with the box counting dimension of $K^s_t$, resp. $K^u_t$. 
\end{remark}

\begin{remark}\label{r.tA'} The first part of Theorem \ref{t.A} (i.e., the continuity of $d_s(t)$ and $d_u(t)$) still holds in the broader context of non-conservative diffeomorphisms: cf. Remark \ref{r.general-continuity} below. On the other hand, our proof of the second part of Theorem \ref{t.A} crucially relies on the conservativeness assumptions. 
\end{remark}

\section{Proof of the main result} 

In plain terms, our strategy of proof of Theorem \ref{t.A} is very similar to \cite{Mo1}: we want to approximate from inside $K^u_t$ and $K^s_t$ by dynamically defined Cantor sets, resp. $\Lambda_t$ by subhorseshoes of $\Lambda$ without losing too much Hausdorff dimension in such a way that the values of $f$ on these approximating objects are controlled from above. 

\subsection{Some preliminaries} Recall that the geometrical description of $\Lambda$ in terms of the Markov partition $\{R_a\}_{a\in\mathcal{A}}$ has a combinatorial counterpart in terms of the Markov shift $\Sigma=\Sigma_{\mathcal{T}}\subset \mathcal{A}^{\mathbb{Z}}$. In particular, we have a homeomorphism $h:\Lambda\to\Sigma$ conjugating $\varphi$ to the shift map $\sigma((a_n)_{n\in\mathbb{Z}})=(a_{n+1})_{n\in\mathbb{Z}}$ and, moreover, we can use $h$ to transfer the function $f$ from $\Lambda$ to a function (still denoted $f$) on $\Sigma$. In this setting, $h(\Lambda_t)=\Sigma_t$ where 
$$\Sigma_t=\{\theta\in\Sigma: \sup\limits_{n\in\mathbb{Z}} f(\sigma^n(\theta))\leq t\}.$$

Given an admissible finite sequence $\alpha=(a_1,\dots, a_n)\in\mathcal{A}^n$ (i.e., $(a_i, a_{i+1})\in\mathcal{T}$ for all $i=1, \dots, n-1$), we define 
$$I^u(\alpha)=\{x\in K^u: g_u^i(x)\in I^u(a_i,a_{i+1}) \, \forall \, i=1, \dots, n-1\}.$$ 
Similarly, given an admissible finite sequence $\alpha=(a_1,\dots, a_n)\in\mathcal{A}^n$, we define 
$$I^s(\alpha^T)=:\{y\in K^s: g_s^i(y)\in I^s(a_i,a_{i-1}) \, \forall \, i=2, \dots, n\}.$$
Here, $\alpha^T=(a_n,\dots, a_1)$ denotes the \emph{transpose} of $\alpha$.  

We say that the \emph{unstable size} $s^{(\textbf{u})}(\alpha)$ of $\alpha$ is the lenght $|I^u(\alpha)|$ of the interval $I^u(\alpha)$ and the \emph{unstable scale} of $\alpha$ is $r^{(\textbf{u})}(\alpha)=\lfloor\log(1/s^{(\textbf{u})}(\alpha))\rfloor$. Similarly, the \emph{stable size} $s^{(\textbf{s})}(\alpha)$ is $s^{(\textbf{s})}(\alpha)=|I^s(\alpha^T)|$ and the \emph{stable scale} of $\alpha$ is $r^{(\textbf{s})}(\alpha)=\lfloor\log(1/s^{(\textbf{s})}(\alpha))\rfloor$.

\begin{remark}\label{r.bounded-distortion} In our context of $C^{1+\varepsilon}$-dynamically defined Cantor sets, we can relate the unstable and stable sizes of $\alpha$ to its length as a word in the alphabet $\mathcal{A}$ via the so-called \emph{bounded distortion property} saying that there exists a constant $c_1=c_1(\varphi,\Lambda)>0$ such that:
$$e^{-c_1}\leq\frac{|I^u(\alpha\beta)|}{|I^u(\alpha)|\cdot |I^u(\beta)|}\leq e^{c_1}, \quad e^{-c_1}\leq\frac{|I^s((\alpha\beta)^T)|}{|I^s((\alpha)^T)|\cdot |I^s((\beta)^T)|}\leq e^{c_1}$$
We refer the reader to \cite[p. 59]{PT} for more details.
\end{remark}

\begin{remark}\label{r.stable-unstable-sizes} In our context of horseshoes of conservative $C^2$ surface diffeomorphisms, there exists a constant $c_2=c_2(\varphi,\Lambda)>0$ such that the stable and unstable sizes of any word $\alpha=(a_1,\dots,a_n)$ in the alphabet $\mathcal{A}$ satisfy 
$$e^{-c_2} |I^s(\alpha^T)|\leq |I^u(\alpha)|\leq e^{c_2} |I^s(\alpha^T)|.$$
Indeed, this happens because $\varphi$ maps the unstable rectangle $R^u(\alpha):=\{x\in R_{a_0}: f^i(x)\in R_{a_i} \forall\, 1\leq i\leq n\}$ diffeomorphically onto the stable rectangle $R^s(\alpha^T):=\{y\in R_{a_n}: f^j(y)\in R_{a_{n-j}} \forall\, 1\leq j\leq n\}$, $\varphi$ preserves areas, and the areas of $R^u(\alpha)$ and $R^s(\alpha^T)$ are comparable to $|I^u(\alpha)|$ and $|I^s(\alpha^T)|$ up to  multiplicative factors. 
\end{remark}

Given $r\in\mathbb{N}$, we define 
$$P^{(\textbf{u})}_r=\{\alpha=(a_1,\dots, a_n)\in\mathcal{A}^n \textrm{ admissible }: r^{(\textbf{u})}(\alpha)\geq r \textrm { and } r^{(\textbf{u})}(a_1\dots a_{n-1})<r\},$$
resp. 
$$P^{(\textbf{s})}_r=\{\alpha=(a_1,\dots, a_n)\in\mathcal{A}^n \textrm{ admissible }: r^{(\textbf{s})}(\alpha)\geq r \textrm { and } r^{(\textbf{s})}(a_1\dots a_{n-1})<r\},$$
and we consider the sets 
$$\mathcal{C}_\textbf{u}(t,r)=\{\alpha\in P_r^{(\textbf{u})}: I^u(\alpha)\cap K_t^u\neq\emptyset\},$$
resp.
$$\mathcal{C}_\textbf{s}(t,r)=\{\alpha\in P_r^{(\textbf{s})}: I^s(\alpha^T)\cap K_t^s\neq\emptyset\}$$ 
whose cardinalities are denoted by $N_\textbf{u}(t,r):=\#\mathcal{C}_\textbf{u}(t,r)$ and $N_\textbf{s}(t,r):=\#\mathcal{C}_\textbf{s}(t,r)$.

\begin{lemma}\label{l.Nu-submultiplicative} For each $t\in\mathbb{R}$, the sequences $N_\textbf{u}(t,r)$, $r\in\mathbb{N}$, and $N_\textbf{s}(t,r)$, $r\in\mathbb{N}$, are essentially submultiplicative in the sense that there exists a constant $c_3=c_3(\varphi, \Lambda)\in\mathbb{N}$ such that 
$$N_{\textbf{u}}(t,n+m)\leq \#\mathcal{A}^{c_3}\cdot N_\textbf{u}(t,n)\cdot N_\textbf{u}(t,m)$$
and 
$$N_\textbf{s}(t,n+m)\leq \#\mathcal{A}^{c_3}\cdot N_\textbf{s}(t,n)\cdot N_\textbf{s}(t,m)$$
for all $n, m\in\mathbb{N}$.
\end{lemma}

\begin{proof} By symmetry (i.e., exchanging the roles of $\varphi$ and $\varphi^{-1}$), it suffices to show that the sequence $N_u(t,r)$, $r\in\mathbb{N}$, is essentially submultiplicative. 

Since the expanding map $g_u$ defining the Cantor set $K^u$ is $C^{1+\varepsilon}$, the usual bounded distortion property (cf. Remark \ref{r.bounded-distortion}) ensures the existence of a constant $c_1=c_1(\varphi,\Lambda)$ such that the sizes of the intervals $I^u(.)$ behave essentially submultiplicatively under admissible concatenations of words, i.e., 
$$|I^u(\alpha\beta \gamma)|\leq e^{2c_1}|I^u(\alpha)|\cdot |I^u(\beta)|\cdot |I^u(\gamma)|$$
for all $\alpha$, $\beta$, $\gamma$ finite words such that the concatenation $\alpha\beta \gamma$ is admissible. 

Next, we observe that, if $\gamma=\gamma_1\dots\gamma_c$ is a finite word in the letters $\gamma_i\in\mathcal{A}$, $1\leq i\leq c$, then 
$$|I^u(\gamma)|\leq \frac{1}{\mu^c}\max_{a\in\mathcal{A}}|I^u_a|$$
where $\mu=\mu_{\textbf{u}}:=\min |Dg_u|>1$, 

Now, we note that, for each $c\in\mathbb{N}$, one can cover $K_t^u$ with $\leq\#\mathcal{A}^{c}\cdot N_\textbf{u}(t,n)\cdot N_\textbf{u}(t,m)$ intervals $I^u(\alpha\beta \gamma)$ with $\alpha\in \mathcal{C}_\textbf{u}(t,n)$, $\beta\in \mathcal{C}_\textbf{u}(t,m)$, $\gamma\in\mathcal{A}^{c}$ and $\alpha\beta \gamma$ admissible.

Therefore, by taking\footnote{In what follows, $\lceil x\rceil$ denotes the smallest integer greater than or equal to $x\in\mathbb{R}$.} 
$$c_3=c_3(\varphi,\Lambda)=\lceil\frac{\log(e^{2c_1}\max_{a\in\mathcal{A}}|I^u_a|)}{\log\mu}\rceil\in\mathbb{N},$$
it follows that we can cover $K_t^u$ with $\leq\#\mathcal{A}^{c_3}\cdot N_\textbf{u}(t,n)\cdot N_\textbf{u}(t,m)$ intervals $I^u(\alpha\beta \gamma)$ whose scales satisfy 
$$r^{(\textbf{u})}(\alpha\beta \gamma)\geq r^{(\textbf{u})}(\alpha)+r^{(\textbf{u})}(\beta)\geq n+m$$
whenever $\alpha\in \mathcal{C}_\textbf{u}(t,n)$, $\beta\in \mathcal{C}_\textbf{u}(t,m)$, $\gamma\in\mathcal{A}^{c_3}$ and $\alpha\beta \gamma$ is admissible. Hence, we conclude that 
$$N_\textbf{u}(t,n+m)\leq \#\mathcal{A}^{c_3}\cdot N_\textbf{u}(t,n)\cdot N_\textbf{u}(t,m)$$
for all $n, m\in\mathbb{N}$.
\end{proof}

From this lemma we get the following immediate corollary:

\begin{corollary}\label{c.Du-definition}For each $t\in\mathbb{R}$, the limits
$$\lim\limits_{m\to\infty}\frac{1}{m}\log N_\textbf{u}(t, m) \quad \textrm{ and } \quad \lim\limits_{m\to\infty}\frac{1}{m}\log N_\textbf{s}(t, m)$$
exist and they coincide with 
$$D_u(t):=\inf\limits_{m\in\mathbb{N}}\frac{1}{m}\log(\#\mathcal{A}^{c_3}\cdot N_\textbf{u}(t,m)) \quad \textrm{ and } \quad D_s(t):=\inf\limits_{m\in\mathbb{N}}\frac{1}{m}\log(\#\mathcal{A}^{c_3} \cdot N_\textbf{u}(t,m))$$
where $c_3=c_3(\varphi,\Lambda)\in\mathbb{N}$ is the constant introduced in Lemma \ref{l.Nu-submultiplicative}.
\end{corollary}

\begin{remark}\label{r.Du} It is not hard to check that $D_u(t)$, resp. $D_s(t)$, coincides with the limit capacity (box counting dimension) of $K_t^u$, resp. $K_t^s$. 
\end{remark}

\subsection{Upper semicontinuity of $D_u(t)$ and $D_s(t)$} In this subsection we show the upper semicontinuity of the limit capacities of $K^u_t$ and $K^s_t$: 
\begin{proposition}\label{p.Du-upper-sc}Using the notation of Corollary \ref{c.Du-definition}, the functions $t\mapsto D_u(t)$ and $t\mapsto D_s(t)$ are upper semicontinuous. 
\end{proposition}

\begin{proof} By symmetry (i.e., exchanging $\varphi$ by $\varphi^{-1}$), our task consists to prove that, for each $t_0\in\mathbb{R}$, the values $D_u(t)$ converge to $D_u(t_0)$ as $t>t_0$ approaches $t_0$. 

By contradiction, suppose that this is not the case. Then, there exists $\eta>0$ such that 
$$D_u(t)>D_u(t_0)+\eta$$
for all $t>t_0$. By Corollary \ref{c.Du-definition}, this implies that 
$$\frac{1}{m}\log(\#\mathcal{A}^{c_3}\cdot N_{\textbf{u}}(t,m))>D_u(t_0)+\eta$$
for all $t>t_0$ and $m\in\mathbb{N}$.

On the other hand, by compactness, for each $m\in\mathbb{N}$, one has 
$$\mathcal{C}_{\textbf{u}}(t_0,m)=\bigcap_{t>t_0}\mathcal{C}_{\textbf{u}}(t,m).$$
In particular, for each $m\in\mathbb{N}$, there exists $t(m)>t_0$ such that $N_{\textbf{u}}(t(m),m)=N_{\textbf{u}}(t_0,m)$. 

Therefore, by putting these facts together, we would deduce that, for each $m\in\mathbb{N}$, 
$$\frac{1}{m}\log(\#\mathcal{A}^{c_3}\cdot N_{\textbf{u}}(t_0,m))=\frac{1}{m}\log(\#\mathcal{A}^{c_3}\cdot N_{\textbf{u}}(t(m),m))>D_u(t_0)+\eta$$
Hence, by letting $m\to\infty$, we would conclude that 
$$D_u(t_0)>D_u(t_0)+\eta,$$
a contradiction.
\end{proof}

\subsection{Construction of $\mathcal{R}_{\varphi,\Lambda}$}\label{ss.R-construction} Fix $r\geq 1$. We define 
$$\mathcal{R}_{\varphi, \Lambda}=\{f\in C^r(M,\mathbb{R}): \grad f(x) \textrm{ is not perpendicular to } E^s_x \textrm{  or } E^u_x \textrm{ for all } x\in\Lambda\}.$$ In other terms, $\mathcal{R}_{\varphi, \Lambda}$ is the class of $C^r$-functions $f:M\to\mathbb{R}$ that are locally monotone along stable and unstable directions. 

\begin{proposition}\label{p.R-generic} Fix $r\geq 2$. If the horseshoe $\Lambda$ has Hausdorff dimension $$\textrm{dim}(\Lambda)<1,$$ then $\mathcal{R}_{\varphi, \Lambda}$ is $C^r$-open and dense in $C^r(M,\mathbb{R})$.
\end{proposition}

\begin{proof}[Proof of Proposition \ref{p.R-generic}] For the proof of this proposition we will need the following well-known fact (cf. \cite[p. 162--165]{PT} and also \cite{HPS}):

\begin{lemma}\label{l.Lipschitz-directions} The maps $\Lambda\ni x\mapsto E_x^s$ and $\Lambda\ni x\mapsto E_x^u$ are $C^1$. 
\end{lemma}


Observe that, by definition, the set $\mathcal{R}_{\varphi,\Lambda}$ is $C^r$-open. Hence, our task is reduced to show that $\mathcal{R}_{\varphi,\Lambda}$ is $C^r$-dense. 

For this sake, let us fix a smooth system of coordinates $z$ on the initial Markov partition $\{R_a\}_{a\in\mathcal{A}}$ so that, given $f\in C^r(M,\mathbb{R})$ and $\varepsilon>0$, the function $g(z)=f(z)+\langle v, z\rangle$ for $z\in\bigcup\limits_{a\in\mathcal{A}}R_a$ can be extended (via an adequate partition of unity) to a $\varepsilon$-$C^r$-perturbation of $f$  whenever $v\in\mathbb{R}^2$ has norm $\|v\|\leq \varepsilon$. 

Suppose that $f\notin\mathcal{R}_{\varphi,\Lambda}$. Given $0<\varepsilon<1$, we will construct $v\in\mathbb{R}^2$ such that $\|v\|\leq \varepsilon$ and $g(z)=f(z)+\langle v, z\rangle\in\mathcal{R}_{\varphi,\Lambda}$.

For each $\delta>0$, let us consider the set $\mathcal{C}(\delta)$ of admissible finite words of the form $(a_{-m},\dots,a_0,\dots, a_{n})$, $m,n\in\mathbb{N}$, such that the rectangle $R(a_{-m},\dots,a_0,\dots, a_{n})=\bigcap\limits_{j=-m}^{n}\varphi^{-j}(R_{a_j})$ has diameter $\leq\delta$ but the rectangles $R(a_{-m},\dots,a_0,\dots, a_{n-1})$ and $R(a_{-m+1},\dots,a_0,\dots, a_{n})$ have diameters $>\delta$. 

Since $\Lambda$ is a horseshoe associated to a $C^2$-diffeomorphism $\varphi$, we know that, for each $d>\textrm{dim}(\Lambda)$, there exists $\delta_0=\delta_0(d)$ such that 
$$\#\mathcal{C}(\delta)\leq1/\delta^d$$
for all $0<\delta<\delta_0$.

By assumption, $\textrm{dim}(\Lambda)<1$. Thus, we can fix once and for all $\textrm{dim}(\Lambda)<d<1$ (e.g., $d=\sqrt{\textrm{dim}(\Lambda)}$) and the corresponding quantity  $\delta_0=\delta_0(d)>0$. 

Next, let us take $\rho=\rho(f,\Lambda)\geq 1$ such that $|\nabla f(z)|\leq\rho$ for all $z\in\Lambda$. Also, since $f$ is $C^2$, there exists a constant $c_4=c_4(f,\Lambda)>0$ such that 
$$|\nabla f(z)-\nabla f(w)|\leq c_4|z-w|$$
for all $z,w\in\Lambda$. Moreover, by Lemma \ref{l.Lipschitz-directions}, there exists a constant $c_5=c_5(\varphi,\Lambda)>0$ such that 
$$|v^s(z)-v^s(w)|\leq c_5|z-w|$$
and 
$$|v^u(z)-v^u(w)|\leq c_5|z-w|$$
where $v^s$ and $v^u$ (resp.) are unitary vectors in $E^s$ and $E^u$ (resp.).

In this setting, for each rectangle $R(\alpha)$, $\alpha\in\mathcal{C}(\delta)$, such that $\nabla f(z_{\alpha})$ is perpendicular to $v^{\ast}(z_{\alpha})$, $\ast=s \textrm{ or } u$, for some $z_{\alpha}\in R(\alpha)$, we write 
\begin{eqnarray*}
\langle\nabla f(z)+v,v^{\ast}(z)\rangle &=& \langle\nabla f(z)-\nabla f(z_{\alpha}),v^{\ast}(z)\rangle + \langle\nabla f(z_{\alpha}),v^{\ast}(z)-v^{\ast}(z_{\alpha})\rangle \\ &+& \langle\nabla f(z_{\alpha}),v^{\ast}(z_{\alpha})\rangle + \langle v,v^{\ast}(z)-v^{\ast}(z_{\alpha})\rangle+ \langle v,v^{\ast}(z_{\alpha})\rangle
\end{eqnarray*}
We control this quantity as follows. Since $R(\alpha)$ has diameter $\leq\delta$, we see from our previous discussion that 
$$|\langle\nabla f(z)-\nabla f(z_{\alpha}),v^{\ast}(z)\rangle|\leq c_4\delta,\quad \quad |\langle\nabla f(z_{\alpha}),v^{\ast}(z)-v^{\ast}(z_{\alpha})\rangle|\leq\rho c_5\delta$$
$$\langle\nabla f(z_{\alpha}),v^{\ast}(z_{\alpha})\rangle=0,\quad \quad |\langle v, v^{\ast}(z)-v^{\ast}(z_{\alpha})\rangle|\leq\|v\| c_5\delta$$
In particular, it follows that
\begin{eqnarray*}
S_{\alpha}&:=&\{v\in\mathbb{R}^2:\|v\|=\varepsilon, \langle\nabla f(z)+v,v^{\ast}(z)\rangle=0 \textrm{ for some } z\in R(\alpha)\} \\ &\subset&\{v\in\mathbb{R}^2:\|v\|=\varepsilon, |\langle v,v^{\ast}(z_{\alpha})\rangle|\leq (c_4+(\rho+1) c_5)\delta\}
\end{eqnarray*}
for $0<\varepsilon<1$.

Therefore, since the function $g(z)=f(z)+\langle v,z\rangle$ satisfies $\nabla g(z)=\nabla f(z)+v$, the proof of the proposition is complete once we show that there exists $v\in\mathbb{R}^2$ such that $\|v\|=\varepsilon$ and $v\notin\bigcup\limits_{\alpha\in\mathcal{C}(\delta)} S_{\alpha}$. As it turns out, this fact is not hard to check: by our previous discussion, for all $0<\delta<\delta_0$, the relative Lebesgue measure of $\bigcup\limits_{\alpha\in\mathcal{C}(\delta)} S_{\alpha}$ is 
$$\leq (c_4+(\rho+1) c_5)\delta^{1-d}/\varepsilon$$ 
because, for each $\alpha\in\mathcal{C}(\delta)$, the relative Lebesgue measure of $S_{\alpha}$ is $\leq (c_4+(\rho+1) c_5)\delta/\varepsilon$, and the cardinality of $\mathcal{C}(\delta)$ is $\leq \delta^{-d}$.
\end{proof}

\subsection{Approximation of $\Lambda_t$ by subhorseshoes}

During this entire subsection we fix a function $f\in\mathcal{R}_{\varphi, \Lambda}$ where $\mathcal{R}_{\varphi, \Lambda}$ was defined in Subsection \ref{ss.R-construction} above. 

By definition of $\mathcal{R}_{\varphi, \Lambda}$, we can refine the initial Markov partition $\{R_a\}_{a\in\mathcal{A}}$ (if necessary) so that the restriction of $f$ to each of the intervals $\{i_a^s\}\times I_a^u$, $a\in\mathcal{A}$, is monotone (i.e., strictly increasing or decreasing), and, furthermore, for some constant $c_6=c_6(\varphi, f)>0$, the following estimates hold: 
\begin{eqnarray}\label{e.c1}
|f(\underline{\theta}^{(1)}a_1\dots a_n a_{n+1}\underline{\theta}^{(3)})-f(\underline{\theta}^{(1)}a_1\dots a_n a_{n+1}'\underline{\theta}^{(4)})| > c_6 |I^u(a_1\dots a_n)|, \\
|f(\underline{\theta}^{(1)}a_{m+1} a_m\dots a_1\underline{\theta}^{(3)})-f(\underline{\theta}^{(2)}a_{m+1}' a_m\dots a_1\underline{\theta}^{(3)})| > c_6 |I^s(a_m\dots a_1)| \nonumber
\end{eqnarray}
whenever $a_{n+1}\neq a_{n+1}'$, $a_{m+1}\neq a_{m+1}'$ and $\underline{\theta}^{(1)}, \underline{\theta}^{(2)}\in\mathcal{A}^{\mathbb{Z}^-}$, $\underline{\theta}^{(3)}, \underline{\theta}^{(4)}\in\mathcal{A}^{\mathbb{N}}$ are admissible. 

Moreover, we observe that, since $f$ is Lipschitz (actually $f\in C^2$), there exists $c_7=c_7(\varphi, f)>0$ such that one also has the following estimates:
\begin{eqnarray}\label{e.c2}
|f(\underline{\theta}^{(1)}a_1\dots a_n a_{n+1}\underline{\theta}^{(3)})-f(\underline{\theta}^{(1)}a_1\dots a_n a_{n+1}'\underline{\theta}^{(4)})| < c_7 |I^u(a_1\dots a_n)|, \\
|f(\underline{\theta}^{(1)}a_{m+1} a_m\dots a_1\underline{\theta}^{(3)})-f(\underline{\theta}^{(2)}a_{m+1}' a_m\dots a_1\underline{\theta}^{(3)})| < c_7 |I^s(a_1\dots a_m)| \nonumber
\end{eqnarray}   
whenever $a_{n+1}\neq a_{n+1}'$, $a_{m+1}\neq a_{m+1}'$ and $\underline{\theta}^{(1)}, \underline{\theta}^{(2)}\in\mathcal{A}^{\mathbb{Z}^-}$, $\underline{\theta}^{(3)}, \underline{\theta}^{(4)}\in\mathcal{A}^{\mathbb{N}}$ are admissible.

The proof of Theorem \ref{t.A} is based on Proposition \ref{p.Du-upper-sc}, Proposition \ref{p.R-generic} and the following statement whose proof will occupy the remainder of this subsection. 

\begin{proposition}\label{p.tA} Let $\Lambda$ be a horseshoe of a conservative $C^2$-diffeomorphism $\varphi$ of a surface $M$. Let $f\in\mathcal{R}_{\varphi, \Lambda}$ with $\mathcal{R}_{\varphi,\Lambda}$ as defined in Subsection \ref{ss.R-construction}, and let us fix $t\in\mathbb{R}$ with $D_{\textbf{u}}(t)>0$, resp. $D_{\textbf{s}}(t)>0$. 

Then, for each $0<\eta<1$, there exists $\delta>0$ and a complete subshift $\Sigma(\mathcal{B}_u)\subset \Sigma\subset\mathcal{A}^{\mathbb{Z}}$, resp. $\Sigma(\mathcal{B}_s)\subset \Sigma\subset\mathcal{A}^{\mathbb{Z}}$, associated to a finite set $\mathcal{B}_u=\{\beta_1^{(u)},\dots,\beta_m^{(u)}\}$, resp. $\mathcal{B}_s=\{\beta_1^{(s)},\dots,\beta_n^{(s)}\}$, of finite sequences $\beta_i^{(u)}=(a_1^{(i); u},\dots,a_{m_i}^{(i); u})\in\mathcal{A}^{m_i}$, resp. $\beta_i^{(s)}=(a_1^{(i); s},\dots,a_{n_i}^{(i); s})\in\mathcal{A}^{n_i}$ such that 
$$\Sigma(\mathcal{B}_u)\subset\Sigma_{t-\delta},\quad \textrm{resp. } \Sigma(\mathcal{B}_s)\subset\Sigma_{t-\delta},$$ 
and 
$$\textrm{dim}(K^u(\Sigma(\mathcal{B}_u)))>(1-\eta)D_u(t), \quad \textrm{dim}(K^s(\Sigma(\mathcal{B}_u^T)))>(1-\eta)D_u(t)$$
resp.
$$\textrm{dim}(K^s(\Sigma(\mathcal{B}_s)))>(1-\eta)D_s(t), \quad \textrm{dim}(K^u(\Sigma(\mathcal{B}_s^T)))>(1-\eta)D_s(t)$$
where $K^u(\Sigma(\ast))$, resp. $K^s(\Sigma(\ast))$, is the subset of $K^u$, resp. $K^s$, consisting of points whose trajectory under $g_u$, resp. $g_s$, follows an itinerary obtained from the concatenation of words in the alphabet $\ast$, and $\ast^T$ is the alphabet whose words are the transposes of the words of the alphabet $\ast$. 

In particular, $D_s(t)=D_u(t)=d_u(t)=d_s(t)$ for all $t\in\mathbb{R}$. 
\end{proposition}

\begin{remark}\label{r.general-continuity} A close inspection of the proof of Proposition \ref{p.tA} reveals that even if $\varphi$ is \emph{not} necessarily conservative, for each $\eta>0$, we can find $\delta>0$ and a complete subshift $\Sigma(\mathcal{B}_u)\subset \Sigma\subset\mathcal{A}^{\mathbb{Z}}$, resp. $\Sigma(\mathcal{B}_s)\subset \Sigma\subset\mathcal{A}^{\mathbb{Z}}$ such that $\textrm{dim}(K^u(\Sigma(\mathcal{B}_u)))>(1-\eta)D_u(t)$ and $\textrm{dim}(K^s(\Sigma(\mathcal{B}_s)))>(1-\eta)D_s(t)$. In particular, we can use this result together with Proposition \ref{p.Du-upper-sc} to get the continuity statement in Remark \ref{r.tA'}. 
\end{remark}

\begin{remark}\label{r.symmetry-du-ds} By symmetry (i.e., exchanging the roles of $\varphi$ and $\varphi^{-1}$), it suffices to exhibit $\mathcal{B}_u$ satisfying the conclusion of Proposition \ref{p.tA}.
\end{remark} 

The construction of $\mathcal{B}_u$ depends on the following three combinatorial lemmas (cf. Lemmas \ref{l.most-good}, \ref{l.excellent-words} and \ref{l.nice-cuts-of-excellent-words} below).

Take $\tau=\eta/100$ and choose $r_0=r_0(\varphi, f, t, \eta)\in\mathbb{N}$ large so that
\begin{equation}\label{e.r0-constraint}
\left|\frac{\log N_u(t,r)}{r}-D_u(t)\right|<\frac{\tau}{2}D_u(t)
\end{equation}
for all $r\in\mathbb{N}$, $r\geq r_0$.  

We set $\mathcal{B}_0=\mathcal{C}_{\textbf{u}}(t,r_0)$, $N_0:=N_{\textbf{u}}(t,r_0)=\#\mathcal{C}_{\textbf{u}}(t,r_0)$, $k:=8 N_0^2\lceil2/\tau\rceil$ and  
$$\widetilde{\mathcal{B}}=\widetilde{\mathcal{B}}_u:=\{\beta:=\beta_1\dots\beta_k : \beta_j\in\mathcal{B}_0 \,\,\, \forall \, 1\leq j\leq k \,\, \textrm{ and } \,\, K_t^u\cap I^u(\beta)\neq\emptyset\}$$

Our plan towards the proof of Proposition \ref{p.tA} is to extract from $\widetilde{\mathcal{B}}$ a rich alphabet $\mathcal{B}$ inducing a complete shift $\Sigma(\mathcal{B})=\mathcal{B}^{\mathbb{Z}}\subset\Sigma_{t-\delta}$ for some $\delta>0$. In this direction, the following notion plays a key role:

\begin{definition}\label{d.good-position} Given $\beta=\beta_1\dots\beta_k\in\widetilde{\mathcal{B}}$ with $\beta_i\in\mathcal{B}_0$ for all $1\leq i\leq k$, we say that $j\in\{1,\dots,k\}$ is a \emph{right-good position} of $\beta$ if there are two elements 
$$\beta^{(n)}=\beta_1\dots\beta_{j-1}\beta_j^{(n)}\dots\beta_k^{(n)}, \quad n=1, 2$$
of $\widetilde{\mathcal{B}}$ such that $\sup I^u(\beta_j^{(1)})<\inf I^u(\beta_j)\leq \sup I^u(\beta_j)<\inf I^u(\beta_j^{(2)})$, i.e., the interval $I^u(\beta_j)$ is located between $I^u(\beta_j^{(1)})$ and $I^u(\beta_j^{(2)})$. 

Similarly, we say that $j\in\{1,\dots, k\}$ is a \emph{left-good position} of $\beta$ if there are two elements 
$$\beta^{(n)}=\beta_1^{(n)}\dots\beta_j^{(n)}\beta_{j+1}\dots\beta_k, \quad n=3, 4$$
of $\widetilde{\mathcal{B}}$ such that $\sup I^s((\beta_j^{(3)})^T)<\inf I^s(\beta_j^T)\leq \sup I^s(\beta_j^T)<\inf I^s((\beta_j^{(4)})^T)$, i.e., the interval $I^s(\beta_j^T)$ is located between $I^s((\beta_j^{(3)})^T)$ and $I^s((\beta_j^{(4)})^T)$.

Finally, we say that $j\in\{1,\dots, k\}$ is a \emph{good position} of $\beta$ if it is both a right-good and a left-good position of $\beta$.
\end{definition}

Our first combinatorial lemma says that most positions of most words of $\widetilde{\mathcal{B}}$ are good:

\begin{lemma}\label{l.most-good} The subset 
$$\mathcal{E}=\{\beta=\beta_1\dots\beta_k\in\widetilde{\mathcal{B}}: \textrm{ the number of good positions of } \beta \textrm{ is } \geq 9k/10\}$$
has cardinality 
$$\#\mathcal{E}\geq\#\widetilde{\mathcal{B}}/2>N_0^{(1-\tau)k}$$
\end{lemma}

\begin{proof} Let us begin by estimating the cardinality of $\widetilde{\mathcal{B}}$. Recall that the sizes of the intervals $I^u(\alpha)$ behave essentially submultiplicatively due the bounded distortion property of $g_u$ (cf. Remark \ref{r.bounded-distortion}) so that, for some constant $c_1=c_1(\varphi, \Lambda)\geq1$, one has   
$$|I^u(\beta)|\leq \exp(-k(r_0-c_1))$$
for any $\beta\in\widetilde{\mathcal{B}}$, and, thus, $\{I^u(\beta) : \beta\in\widetilde{\mathcal{B}}\}$ is a covering of $K_t^u$ by intervals of sizes $\leq \exp(-k(r_0-c_1))$. In particular, we have a natural surjective map $h:\widetilde{\mathcal{B}}\to \mathcal{C}_{\textbf{u}}(t, k(r_0-c_1))$ given by 
$$h((b_1\dots b_{n(k)}))=(b_1\dots b_j)$$ where  
$$j:=\min\{1\leq i\leq n(k): r^{(\textbf{u})}(b_1\dots b_i)\geq k(r_0-c_1)\}$$

On the other hand, since $k(r_0-c_1)\geq r_0$ for $r_0=r_0(\varphi, f, t)\in\mathbb{N}$ large enough, by \eqref{e.r0-constraint} we have that
$$\#\mathcal{C}_{\textbf{u}}(t, k(r_0-c_1)):=N_{\textbf{u}}(t, k(r_0-c_1))\geq \frac{1}{4}\exp(k(r_0-c_1)D_u(t))$$

In particular, by putting these informations together, we see that, for $r_0=r_0(\varphi, f, t)\in\mathbb{N}$ large enough, the following estimate holds: 
\begin{eqnarray*}
\#\widetilde{\mathcal{B}}&\geq& \frac{1}{4}\exp(k(r_0-c_1)D_u(t)) \\ &>& 2\exp(k(r_0-2c_1)D_u(t)) \quad (\textrm{since } k \textrm{ is large for } r_0 \textrm{ large  by } \eqref{e.r0-constraint} \textrm{ and }D_u(t)>0)\\ 
&\geq& 2 \exp\left(\left(1-\frac{\tau}{2}\right)k r_0 D_u(t)\right)>2 \exp\left((1-\tau)\left(1+\frac{\tau}{2}\right)k r_0 D_u(t)\right) \\ 
&>& 2N_0^{(1-\tau)k} \quad (\textrm{since } N_0<\exp((1+\tau/2)r_0D_u(t)) \textrm{ by }\eqref{e.r0-constraint})
\end{eqnarray*}

In summary, the set $\widetilde{\mathcal{B}}$ has cardinality 
\begin{equation}\label{e.B-tilde-cardinality}
\#\widetilde{\mathcal{B}}>2N_0^{(1-\tau)k}
\end{equation}
for $r_0=r_0(\varphi, f,t,\eta)\in\mathbb{N}$ large enough.

Now, let us estimate the cardinality of the subset of $\widetilde{\mathcal{B}}$ consisting of words $\beta$ such that at least $k/20$ positions are not right-good. First, we notice that there are at most $2^k$ choices for the set of $m\geq k/20$ right-bad (i.e., not right-good) positions. Secondly, once this set of right-bad positions is fixed: 
\begin{itemize}
\item if $j$ is a right-bad position and $\beta_1, \dots, \beta_{j-1}\in\mathcal{B}_0$ were already chosen, then we see that there are at most two possibilities for $\beta_j\in\mathcal{B}_0$ (namely, the choices leading to the leftmost and rightmost subintervals of $I^u(\beta_1\dots\beta_{j-1})$ of the form $I^u(\beta_1\dots\beta_k)$ intersecting $K_t^u$);
\item if $j$ is not a right-bad position, then there are at most $N_0$ choices of $\beta_j$.
\end{itemize}
In particular, once a set of $m\geq k/20$ right-bad positions is fixed, the quantity of words in $\widetilde{\mathcal{B}}$ with this set of $m$ right-bad positions is at most 
$$2^m\cdot N_0^{k-m}\leq 2^{k/20} \cdot N_0^{19k/20}$$
Therefore, the quantity of words in $\widetilde{\mathcal{B}}$ with at least $k/20$ right-bad positions is 
$$\leq 2^k\cdot 2^{k/20}\cdot N_0^{19k/20} = 2^{21k/20}\cdot N_0^{19k/20}$$
Analogously, the quantity of words in $\widetilde{\mathcal{B}}$ with at least $k/20$ left-bad positions is also $\leq 2^{21k/20}\cdot N_0^{19k/20}$.

By definition of $\mathcal{E}$, it follows that 
$$\#\mathcal{E}\geq \#\widetilde{\mathcal{B}}-2\cdot 2^{21k/20}\cdot N_0^{19k/20}$$
Since $\#\widetilde{\mathcal{B}}>2N_0^{(1-\tau)k}$ (by \eqref{e.B-tilde-cardinality}) and $2^{1+21k/20}\cdot N_0^{19k/20}<N_0^{(1-\tau)k}$ (from our choices of $r_0$ large, $N_0$ and $k$), we deduce that 
$$\#\mathcal{E}\geq \frac{1}{2}\#\widetilde{\mathcal{B}}>N_0^{(1-\tau)k}.$$
This completes the proof of the lemma.  
\end{proof}

In the sequel, we will call \emph{excellent word} an arbitrary element $\beta\in\mathcal{E}$ of the subset $\mathcal{E}$ introduced in Lemma \ref{l.most-good} above.

Our second combinatorial lemma states that several excellent words $\beta\in\mathcal{E}$ share the same good poistions and the same words of $\mathcal{B}_0$ appearing in these positions. 

\begin{lemma}\label{l.excellent-words} There are natural numbers $1\leq \widehat{j}_1\leq \dots\leq \widehat{j}_{3N_0^2}\leq k$ with $\widehat{j}_{m+1}-\widehat{j}_m\geq 2\lceil2/\tau\rceil$ for $1\leq m<3N_0^2$ and words $\widehat{\beta}_{\widehat{j}_1}, \widehat{\beta}_{\widehat{j}_1+1}, \dots, \widehat{\beta}_{\widehat{j}_{3N_0^2}}, \widehat{\beta}_{\widehat{j}_{3N_0^2}+1}\in\mathcal{B}_0$ such that the set 
$$X:=\{\beta_1\dots\beta_k\in\mathcal{E}: \widehat{j}_m, \widehat{j}_{m}+1 \textrm{ good positions and } \beta_{\widehat{j}_m}=\widehat{\beta}_{\widehat{j}_m}, \beta_{\widehat{j}_m+1}=\widehat{\beta}_{\widehat{j}_m+1}\,\forall\, 1\leq m<3N_0^2\}$$
has cardinality 
$$\#X>N_0^{(1-2\tau)k}$$
\end{lemma}

\begin{proof} Given $\beta=\beta_1\dots\beta_k\in\mathcal{E}$ an excellent word, we can find $\lceil2k/5\rceil$ positions 
$$1\leq i_1\leq \dots \leq i_{\lceil2k/5\rceil}\leq k$$
such that $i_{n+1}\geq i_n+2$ for all $1\leq n<\lceil 2k/5\rceil$ and the positions
$$i_1, i_1+1, \dots, i_{\lceil2k/5\rceil}, i_{\lceil2k/5\rceil}+1$$
are good. 

Since we took $k:=8N_0^2\lceil2/\tau\rceil$, it makes sense to set  
$$j_n:=i_{n\lceil2/\tau\rceil} \quad \textrm{ for } 1\leq n\leq 3N_0^2$$
because $3N_0^2\lceil2/\tau\rceil<(16/5)N_0^2\lceil2/\tau\rceil = 2k/5$. In this way, we obtain positions such that 
$$j_{n+1}-j_n\geq 2\lceil2/\tau\rceil  \quad \textrm{ for } 1\leq n\leq 3N_0^2$$
and $j_1, j_1+1, \dots, j_{3N_0^2}, j_{3N_0^2}+1$ are good positions.

Now, we observe that: 
\begin{itemize}
\item the number of possibilities for $(j_1,\dots,j_{3N_0^2})$ is at most $\binom{k}{3N_0^2}<2^k$ and 
\item for each fixed $(j_1,\dots,j_{3N_0^2})$, the number of possibilities for $$(\beta_{j_1},\beta_{j_1+1},\dots,\beta_{3N_0^2},\beta_{3N_0^2}+1)$$ is at most $N_0^{6N_0^2}$ (because $N_0=\#\mathcal{B}_0$ and $\beta_{j_n}, \beta_{j_n+1}\in\mathcal{B}_0$ for $1\leq n\leq 3N_0^2$).
\end{itemize}

Hence, we can choose $1<\widehat{j_1}<\dots<\widehat{j}_{3N_0^2}<k$ with $\widehat{j}_{n+1}-\widehat{j}_n\geq 2\lceil2/\tau\rceil$ for $1\leq n\leq 3N_0^2$, and some words 
$$\widehat{\beta}_{\widehat{j}_1}, \widehat{\beta}_{\widehat{j}_1+1},\dots, \widehat{\beta}_{\widehat{j}_{3N_0^2}}, \widehat{\beta}_{\widehat{j}_{3N_0^2}+1}\in\mathcal{B}_0$$
such that the set 
$$X:=\{\beta_1\dots\beta_k\in\mathcal{E}: \widehat{j}_m, \widehat{j}_{m}+1 \textrm{ good positions and } \beta_{\widehat{j}_m}=\widehat{\beta}_{\widehat{j}_m}, \beta_{\widehat{j}_m+1}=\widehat{\beta}_{\widehat{j}_m+1}\,\forall\, 1\leq m<3N_0^2\}$$
has cardinality 
$$\#X\geq\frac{\#\mathcal{E}}{2^k N_0^{6 N_0^2}}$$
Since $\#\mathcal{E}\geq N_0^{(1-\tau)k}$ (cf. Lemma \ref{l.most-good}) and $2^kN_0^{6N_0^2}<N_0^{\tau k}$, the proof of the lemma is complete.
\end{proof}

Our third combinatorial lemma states that it is possible to cut excellent words in the subset $X$ provided by Lemma \ref{l.excellent-words} at certain good positions in such a way that one obtains a finite set $\mathcal{B}=\mathcal{B}_u$ with non-neglectible cardinality. 

\begin{lemma}\label{l.nice-cuts-of-excellent-words} In the context of Lemma \ref{l.excellent-words}, given $1\leq p<q\leq 3N_0^2$, let $\pi_{p,q}:X\to\mathcal{B}_0^{\widehat{j}_q-\widehat{j}_p}$ be the projection 
$$\pi_{p,q}(\beta_1\dots\beta_k):=(\beta_{\widehat{j}_p+1},\dots,\beta_{\widehat{j}_q})$$
obtained by cutting a word $\beta_1\dots\beta_k\in X$ at the positions $\widehat{j}_p$ and $\widehat{j}_q$ and discarding the words $\beta_j$ with $j\leq\widehat{j}_p$ and $j>\widehat{j}_q$.

Then, there are $1\leq p_0<q_0\leq 3N_0^2$ such that $\widehat{\beta}_{\widehat{j}_{p_0}}=\widehat{\beta}_{\widehat{j}_{q_0}}$, $\widehat{\beta}_{\widehat{j}_{p_0+1}}=\widehat{\beta}_{\widehat{j}_{q_0+1}}$ and the cardinality of 
$$\mathcal{B}=\mathcal{B}_u:=\pi_{p_0,q_0}(X)$$
is 
$$\#\mathcal{B}>N_0^{(1-10\tau)(\widehat{j}_{q_0}-\widehat{j}_{p_0})}$$
\end{lemma}

\begin{proof} For each pair of indices $(p,q)$, $1\leq p<q\leq 3N_0^2$ with $\#\pi_{p,q}(X)\leq N_0^{(1-10\tau)(\widehat{j}_q-\widehat{j}_p)}$, we exclude from $\{1,\dots,3N_0^2\}$ the interval $[p,q-1]\cap\mathbb{Z}:=\{p, p+1,\dots, q-1\}$.

We affirm that the cardinality of the set $\mathcal{Z}$ of excluded indices is $<2N_0^{2}$. Indeed, suppose by contradiction that $\#\mathcal{Z}\geq 2N_0^2$. By definition, $\mathcal{Z}$ is the union of the finite family of intervals $[p,q-1]$ with $\#\pi_{p,q}(X)\leq N_0^{(1-10\tau)(\widehat{j}_q-\widehat{j}_p)}$. Now, it is not hard to see that, given any finite family of intervals $\mathcal{F}$, there exists a subfamily $\mathcal{F}_0\subset\mathcal{F}$ of \emph{disjoint} intervals whose sum of lengths 
$$\sum\limits_{J\in\mathcal{F}_0}|J|$$ is at least \emph{half} of the measure of 
$$\bigcup\limits_{I\in\mathcal{F}}I$$
Applying this fact to the family of intervals defining $\mathcal{Z}$, it follows that we can select a subfamily $\mathcal{P}$ of pairs $(p,q)$ such that $\pi_{p,q}(X)\leq N_0^{(1-10\tau)(\widehat{j}_q-\widehat{j}_p)}$ leading to disjoint intervals $[p,q-1]$ whose sum of lengths 
$$\sum\limits_{(p,q)\in\mathcal{P}}(q-p)\geq \frac{1}{2}\#\mathcal{Z}.$$
Since $\widehat{j}_m$, $1\leq m\leq 3N_0^2$, were constructed so that $\widehat{j}_{m+1}-\widehat{j}_m\geq 2\lceil2/\tau\rceil$ (cf. Lemma \ref{l.excellent-words}) and we are assuming (by contradiction) that $\#\mathcal{Z}\geq 2N_0^2$, we deduce from the previous estimate that 
\begin{equation}\label{e.excluded-pairs-1}
\sum\limits_{(p,q)\in\mathcal{P}}(\widehat{j}_q-\widehat{j}_p)\geq2\lceil2/\tau\rceil\sum\limits_{(p,q)\in\mathcal{P}}(q-p)\geq 2\lceil2/\tau\rceil N_0^2
\end{equation}
On the other hand, since $\#\pi_{p,q}(X)\leq N_0^{(1-10\tau)(\widehat{j}_q-\widehat{j}_p)}$ for $(p,q)\in\mathcal{P}$, we get that 
$$
\#X<N_0^{(1-10\tau)\sum\limits_{(p,q)\in\mathcal{P}}(\widehat{j}_q-\widehat{j}_p)}\cdot N_0^{\#\{1\leq i\leq k\,:\, i\notin[\widehat{j}_p, \widehat{j}_q - 1] \,\forall\,(p,q)\in\mathcal{P}\}}
$$
because there are at most $N_0$ choices for $\beta_i\in\mathcal{B}_0$ for indices $i\notin\bigcup\limits_{(p,q)\in\mathcal{P}}[\widehat{j}_p,\widehat{j}_q - 1]$, $1\leq i\leq k$. By plugging \eqref{e.excluded-pairs-1} into the previous estimate, we would obtain that 
$$\#X<N_0^{(1-10\tau)\cdot 2N_0^2\lceil2/\tau\rceil}\cdot N_0^{k-2N_0^2\lceil2/\tau\rceil}=N_0^{k-20\tau N_0^2\lceil2/\tau\rceil}$$
However, by Lemma \ref{l.excellent-words}, we also have $\#X>N_0^{(1-2\tau)k}$, so that 
$$N_0^{(1-2\tau)k}<N_0^{k-20\tau N_0^2\lceil2/\tau\rceil},\quad \textrm{ i.e., } \quad 20 N_0^2\lceil2/\tau\rceil<2k,$$
a contradiction with our choice of $2k=16 N_0^2\lceil2/\tau\rceil<20N_0^2\lceil2/\tau\rceil$.

Once we know that the subset $\mathcal{Z}$ of excluded indices in $\{1,\dots,3N_0^2\}$ has cardinality $\#\mathcal{Z}<2N_0^2$, we conclude that there are at least $N_0^2+1$ non-excluded indices. Since for each $1\leq m\leq 3N_0^2$ there are at most $N_0$ possibilities for $\widehat{\beta}_{\widehat{j}_m}$, we deduce that there are two non-excluded indices $1\leq p_0<q_0\leq 3N_0^2$ such that $\widehat{\beta}_{\widehat{j}_{p_0}}=\widehat{\beta}_{\widehat{j}_{q_0}}$ and $\widehat{\beta}_{\widehat{j}_{p_0}+1}=\widehat{\beta}_{\widehat{j}_{q_0}+1}$. By definition of non-excluded index, $\mathcal{B}=\mathcal{B}_u:=\pi_{p_0,q_0}(X)$ satisfies the conclusions of the Lemma.
\end{proof}

At this point, we are ready to complete the proof of Proposition \ref{p.tA} by showing that the finite set $\mathcal{B}=\mathcal{B}_u=\pi_{p_0,q_0}(X)$ has the desired properties. 

\begin{proof}[Proof of Proposition \ref{p.tA}] Recall from Remark \ref{r.symmetry-du-ds} that our task is reduced to show that the complete shift $\Sigma(\mathcal{B})$ associated to $\mathcal{B}=\mathcal{B}_u$ generates a Cantor set $K^u(\mathcal{B})$ with Hausdorff dimension 
$$\textrm{dim}(K^u(\Sigma(\mathcal{B})))>(1-\eta)D_u(t)$$
and $\Sigma(\mathcal{B})\subset\Sigma_{t-\delta}$ for some $\delta>0$.  

We start by estimating the Hausdorff dimension of the Cantor set $K^u(\Sigma(\mathcal{B}_u))$ induced by $\mathcal{B}_u=\mathcal{B}$. 

Note that $K^u(\Sigma(\mathcal{B}_u))$ is a $C^{1+\varepsilon}$-dynamically defined Cantor set associated to certain iterates of $g_u$ defined on the intervals $I^u(\alpha)$, $\alpha\in\mathcal{B}$. In this situation, it is known from the usual bounded distortion property (cf. Remark \ref{r.bounded-distortion}) that the Hausdorff dimension and box counting dimensions of $K^u(\Sigma(\mathcal{B}))$ coincide and they satisfy  
$$\textrm{dim}(K^u(\Sigma(\mathcal{B})))\geq (1-\frac{\tau}{2})\frac{\log\#\mathcal{B}}{-\log(\min\limits_{\alpha\in\mathcal{B}}|I^u(\alpha)|)}$$ 
for $r_0$ sufficiently large. 

Since $\#\mathcal{B}> N_0^{(1-10\tau)(\widehat{j}_{q_0}-\widehat{j}_{p_0})}$ (cf. Lemma \ref{l.nice-cuts-of-excellent-words}) and, for some constant $c_1=c_1(\varphi,\Lambda)>0$, one has $|I^u(\alpha)|\geq e^{-(\widehat{j}_{q_0}-\widehat{j}_{p_0})(r_0+c_1)}$ (by the usual bounded distortion property cf. Remark \ref{r.bounded-distortion}) for each $\alpha\in\mathcal{B}$, we deduce from the previous estimate that 
$$\textrm{dim}(K^u(\Sigma(\mathcal{B})))>\frac{(1-\frac{\tau}{2})(1-10\tau)\log N_0}{r_0+c_1}.$$
Now, recall that $N_0=N_{\textbf{u}}(t,r_0)$ satisfy \eqref{e.r0-constraint} so that $\log N_0>(1-\tau/2)r_0 D_u(t)$. Hence, by plugging this into the previous inequality (and by recalling that $\tau=\eta/100$), we obtain
$$\textrm{dim}(K^u(\Sigma(\mathcal{B})))>\frac{(1-10\tau)(1-\frac{\tau}{2})^2 r_0}{r_0+c_1} D_u(t)>(1-12\tau)D_u(t)>(1-\eta)D_u(t)$$
for $r_0=r_0(\eta)\in\mathbb{N}$ sufficiently large.

Similarly, we also have that
$$\textrm{dim}(K^s(\Sigma(\mathcal{B}^T)))\geq (1-\frac{\tau}{2})\frac{\log\#\mathcal{B}}{-\log(\min\limits_{\alpha\in\mathcal{B}}|I^s(\alpha^T)|)}$$
Because $|I^s(\alpha^T)|$ is comparable to $|I^u(\alpha)|$ up to the multiplicative factor $e^{c_2}$ (cf. Remark \ref{r.stable-unstable-sizes}), we deduce from the computations of the previous paragraph that 
$$\textrm{dim}(K^s(\Sigma(\mathcal{B}^T)))>\frac{(1-10\tau)(1-\frac{\tau}{2})^2 r_0}{(r_0+c_1+c_2)} D_u(t)>(1-\eta)D_u(t)$$
for $r_0=r_0(\eta)\in\mathbb{N}$ sufficiently large.

At this point, it remains only to prove that $\Sigma(\mathcal{B})\subset\Sigma_{t-\delta}$ for some $\delta>0$. For this sake, we denote by 
$$\gamma_1:=\widehat{\beta}_{\widehat{j}_{p_0}+1}=\widehat{\beta}_{\widehat{j}_{q_0}+1}, \quad \gamma_2:=\widehat{\beta}_{\widehat{j}_{p_0}}=\widehat{\beta}_{\widehat{j}_{q_0}},\quad\textrm{ and } \quad \widehat{n}:=\widehat{j}_{q_0}-\widehat{j}_{p_0}$$
so that, by definition, the elements $\beta\in\mathcal{B}$ have the form
$$\beta=\gamma_1\beta_{\widehat{j}_{p_0}+2}\dots\beta_{\widehat{j}_{p_0}+\widehat{n}-1}\gamma_2$$
with $\gamma_1,\beta_{\widehat{j}_{p_0}+2},\dots,\beta_{\widehat{j}_{p_0}+\widehat{n}-1},\gamma_2\in\mathcal{B}_0$, $\widehat{n}:=\widehat{j}_{q_0}-\widehat{j}_{p_0}$, and 
$$\sup I^s((\gamma_1')^T)<\inf I^s(\gamma_1^T)\leq \sup I^s(\gamma_1^T)<\inf I^s((\gamma_1'')^T),$$
$$\sup I^u(\gamma_2')<\inf I^u(\gamma_2)\leq \sup I^u(\gamma_2)<\inf I^u(\gamma_2''),$$
for some words $\gamma_1',\gamma_1'',\gamma_2',\gamma_2''\in\mathcal{B}_0$ verifying 
$$I^u(\gamma_1'\beta_{\widehat{j}_{p_0}+2}\dots\beta_{\widehat{j}_{p_0}+\widehat{n}-1}\gamma_2\gamma_1)\cap K_t^u\neq\emptyset, \quad I^u(\gamma_1''\beta_{\widehat{j}_{p_0}+2}\dots\beta_{\widehat{j}_{p_0}+\widehat{n}-1}\gamma_2\gamma_1)\cap K_t^u\neq\emptyset,$$
$$I^u(\gamma_2\gamma_1\beta_{\widehat{j}_{p_0}+2}\dots\beta_{\widehat{j}_{p_0}+\widehat{n}-1}\gamma_2')\cap K_t^u\neq\emptyset, \quad I^u(\gamma_2\gamma_1\beta_{\widehat{j}_{p_0}+2}\dots\beta_{\widehat{j}_{p_0}+\widehat{n}-1}\gamma_2'')\cap K_t^u\neq\emptyset$$
In particular, an element of $\Sigma(\mathcal{B})$ has the form 
$$\underline{\theta}^{(1)}\gamma_2;\gamma_1\beta_{\widehat{j}_{p_0}+2}\dots\beta_{\widehat{j}_{p_0}+\widehat{n}-1}\gamma_2\gamma_1\underline{\theta}^{(2)}$$
where $\gamma_1\underline{\theta}^{(2)}\in\mathcal{A}^{\mathbb{N}}$ and $\underline{\theta}^{(1)}\gamma_2\in\mathcal{A}^{\mathbb{Z}^-}$ are infinite concatenations of elements of $\mathcal{B}$ and the symbol $;$ serves to mark the location of the entry of index $1$ of the bi-infinite sequence $\underline{\theta}^{(1)}\gamma_2;\gamma_1\beta_{\widehat{j}_{p_0}+2}\dots\beta_{\widehat{j}_{p_0}+\widehat{n}-1}\gamma_2\gamma_1\underline{\theta}^{(2)}$. 

In this notation, our task of showing that $\Sigma(\mathcal{B})\subset\Sigma_{t-\delta}$ for some $\delta>0$ is equivalent to show that
\begin{equation}\label{e.decreasing-f}f(\sigma^{\ell}(\underline{\theta}^{(1)}\gamma_2;\gamma_1\beta_{\widehat{j}_{p_0}+2}\dots\beta_{\widehat{j}_{p_0}+\widehat{n}-1}\gamma_2\gamma_1\underline{\theta}^{(2)}))\leq t-\delta
\end{equation}
for all $0\leq \ell\leq \widehat{m}_1+\widehat{m}+\widehat{m_2}-1$ where $\gamma_1:=a_1\dots a_{\widehat{m}_1}$, $\beta_{\widehat{j}_{p_0}+2}\dots\beta_{\widehat{j}_{p_0}+\widehat{n}-1}:=b_{1}\dots b_{\widehat{m}}$ and $\gamma_2:=d_{1}\dots d_{\widehat{m}_2}$.

We consider two regimes for $1\leq \ell\leq \widehat{m}_1+\widehat{m}+\widehat{m_2}-1$:
\begin{itemize}
\item[I)] $\widehat{m}_1\leq \ell\leq \widehat{m}_1+\widehat{m}-1$
\item[II)] $0\leq\ell\leq \widehat{m}_1-1$ or $\widehat{m}_1+\widehat{m}\leq\ell\leq \widehat{m}_1+\widehat{m}+\widehat{m}_2-1$.
\end{itemize}

In case I), we write $\ell=\widehat{m}_1-1+j$ so that 
\begin{equation}\label{e.sigma-l-B}\sigma^{\ell}(\underline{\theta}^{(1)}\gamma_2;\gamma_1\beta_{\widehat{j}_{p_0}+2}\dots\beta_{\widehat{j}_{p_0}+\widehat{n}-1}\gamma_2\gamma_1\underline{\theta}^{(2)}) = \underline{\theta}^{(1)}\gamma_2\gamma_1b_1\dots b_{j-1}; b_j\dots b_{\widehat{m}}\gamma_2\gamma_1\underline{\theta}^{(2)}
\end{equation}

We have two possibilities:
\begin{itemize}
\item[I.a)] $|I^s((\gamma_1b_1\dots b_{j-1})^T)|\leq |I^u(b_j\dots b_{\widehat{m}}\gamma_2)|$
\item[I.b)] $|I^u(b_j\dots b_{\widehat{m}}\gamma_2)|\leq |I^s((\gamma_1b_1\dots b_{j-1})^T)|$
\end{itemize}

In case I.a), we choose $\gamma_2^{\ast}\in\{\gamma_2', \gamma_2''\}$ such that 
$$f(\underline{\theta}^{(1)}\gamma_2\gamma_1b_1\dots b_{j-1}; b_j\dots b_{\widehat{m}}\gamma_2\gamma_1\underline{\theta}^{(2)})<f(\underline{\theta}^{(1)}\gamma_2\gamma_1b_1\dots b_{j-1}; b_j\dots b_{\widehat{m}}\gamma_2^{\ast}\underline{\theta}^{(4)})$$
for any $\underline{\theta}^{(4)}\in\mathcal{A}^{\mathbb{N}}$ (because of the local monotonicity of $f$ along stable and unstable manifolds). By \eqref{e.c1}, it follows that  
\begin{eqnarray*}
& &f(\underline{\theta}^{(1)}\gamma_2\gamma_1b_1\dots b_{j-1}; b_j\dots b_{\widehat{m}}\gamma_2\gamma_1\underline{\theta}^{(2)}) + c_6|I^u(b_j\dots b_{\widehat{m}}\gamma_2)| \\ & &< f(\underline{\theta}^{(1)}\gamma_2\gamma_1b_1\dots b_{j-1}; b_j\dots b_{\widehat{m}}\gamma_2^{\ast}\underline{\theta}^{(4)})
\end{eqnarray*}
for some $c_6=c_6(\varphi, f)>0$. On the other hand, by \eqref{e.c2}, we also know that, for some $c_7=c_7(\varphi, f)>0$, the function $f$ obeys the Lipschitz estimate 
\begin{eqnarray*}
& & |f(\underline{\theta}^{(3)}\gamma_2\gamma_1b_1\dots b_{j-1}; b_j\dots b_{\widehat{m}}\gamma_2^{\ast}\underline{\theta}^{(4)}) - f(\underline{\theta}^{(1)}\gamma_2\gamma_1b_1\dots b_{j-1}; b_j\dots b_{\widehat{m}}\gamma_2^{\ast}\underline{\theta}^{(4)})| \\ & & < c_7 |I^s((\gamma_2\gamma_1b_1\dots b_{j-1})^T)|
\end{eqnarray*}
for any $\underline{\theta}^{(3)}\in\mathcal{A}^{\mathbb{Z}^-}$. From these estimates, we obtain that 
$$f(\underline{\theta}^{(1)}\gamma_2\gamma_1b_1\dots b_{j-1}; b_j\dots b_{\widehat{m}}\gamma_2\gamma_1\underline{\theta}^{(2)})+ c_6|I^u(b_j\dots b_{\widehat{m}}\gamma_2)|<$$
$$f(\underline{\theta}^{(3)}\gamma_2\gamma_1b_1\dots b_{j-1}; b_j\dots b_{\widehat{m}}\gamma_2^{\ast}\underline{\theta}^{(4)}) + c_7 |I^s((\gamma_2\gamma_1b_1\dots b_{j-1})^T)|$$
for any $\underline{\theta}^{(3)}\in\mathcal{A}^{\mathbb{Z}^-}$ and $\underline{\theta}^{(4)}\in\mathcal{A}^{\mathbb{N}}$. Now, we observe that the usual bounded distortion property (cf. Remark \ref{r.bounded-distortion}) implies that  
$$|I^s((\gamma_2\gamma_1b_1\dots b_{j-1})^T)|\leq e^{c_1}|I^s(\gamma_2^T)|\cdot |I^s((\gamma_1b_1\dots b_{j-1})^T)|$$
for some $c_1=c_1(\varphi)>0$. By plugging this information into the previous estimate, we have 
 $$f(\underline{\theta}^{(1)}\gamma_2\gamma_1b_1\dots b_{j-1}; b_j\dots b_{\widehat{m}}\gamma_2\gamma_1\underline{\theta}^{(2)})+ c_6|I^u(b_j\dots b_{\widehat{m}}\gamma_2)|<$$
$$f(\underline{\theta}^{(3)}\gamma_2\gamma_1b_1\dots b_{j-1}; b_j\dots b_{\widehat{m}}\gamma_2^{\ast}\underline{\theta}^{(4)}) + c_7 e^{c_1} |I^s(\gamma_2^T)|\cdot |I^s((\gamma_1b_1\dots b_{j-1})^T)|$$
Since we are dealing with case I.a), i.e., $|I^s((\gamma_1b_1\dots b_{j-1})^T)|\leq |I^u(b_j\dots b_{\widehat{m}}\gamma_2)|$, we deduce that 
$$f(\underline{\theta}^{(1)}\gamma_2\gamma_1b_1\dots b_{j-1}; b_j\dots b_{\widehat{m}}\gamma_2\gamma_1\underline{\theta}^{(2)})<$$
$$f(\underline{\theta}^{(3)}\gamma_2\gamma_1b_1\dots a_{j-1}; b_j\dots b_{\widehat{m}}\gamma_2^{\ast}\underline{\theta}^{(4)}) - (c_6- c_7 e^{c_1} |I^s(\gamma_2^T)|)\cdot |I^u(b_j\dots b_{\widehat{m}}\gamma_2)|$$
Next, we note that $c_7 e^{c_1}|I^s(\gamma_2^T)|<c_6/2$ if $r_0=r_0(\varphi, f)\in\mathbb{N}$ is sufficiently large: indeed, since $\gamma_2\in\mathcal{B}_0=\mathcal{C}_u(t,r_0)$, we know that $|I^u(\gamma_2)|\leq e^{-r_0}$; on the other hand, the usual bounded distortion property ensures that $|I^s(\gamma_2^T)|\leq |I^u(\gamma_2)|^{c_8}$ for some constant $c_8=c_8(\varphi)>0$, so that $c_7 e^{c_1}|I^s(\gamma_2^T)|<c_6/2$ for $r_0>(1/c_8)\log(2c_7e^{c_1}/c_6)$. In particular, for $r_0=r_0(\varphi, f, t, \eta)\in\mathbb{N}$ sufficiently large, we have that 
\begin{eqnarray}\label{e.decreasing-f-caseIa}f(\underline{\theta}^{(1)}\gamma_2\gamma_1b_1\dots b_{j-1}; b_j\dots b_{\widehat{m}}\gamma_2\gamma_1\underline{\theta}^{(2)})< \\
f(\underline{\theta}^{(3)}\gamma_2\gamma_1b_1\dots b_{j-1}; b_j\dots b_{\widehat{m}}\gamma_2^{\ast}\underline{\theta}^{(4)}) - (c_6/2)\cdot |I^u(b_j\dots b_{\widehat{m}}\gamma_2)|\nonumber
\end{eqnarray}
for any $\underline{\theta}^{(3)}\in\mathcal{A}^{\mathbb{Z}^-}$ and $\underline{\theta}^{(4)}\in\mathcal{A}^{\mathbb{N}}$. Now, we recall that $\gamma_2^*\in\{\gamma_2', \gamma_2''\}$, so that 
$$I^u(\gamma_2\gamma_1\beta_{\widehat{j}_{p_0}+2}\dots\beta_{\widehat{j}_{p_0}+\widehat{n}-1}\gamma_2^*)\cap K_t^u\neq\emptyset.$$
By definition, this means that there are $\underline{\theta}^{(3)}_*\in\mathcal{A}^{\mathbb{Z}^-}$ and $\underline{\theta}^{(4)}_*\in\mathcal{A}^{\mathbb{N}}$ with 
$$\underline{\theta}^{(3)}_*;\gamma_2\gamma_1\beta_{\widehat{j}_{p_0}+2}\dots\beta_{\widehat{j}_{p_0}+\widehat{n}-1}\gamma_2^*\underline{\theta}^{(4)}_*\in\Sigma_t,$$
and, \emph{a fortiori}, 
$$f(\sigma^{\widehat{m}_2+\ell}(\underline{\theta}^{(3)}_*;\gamma_2\gamma_1b_1\dots b_{\widehat{m}}\gamma_2^{\ast}\underline{\theta}^{(4)}_*))=f(\underline{\theta}^{(3)}_*\gamma_2\gamma_1b_1\dots b_{j-1}; b_j\dots b_{\widehat{m}}\gamma_2^{\ast}\underline{\theta}^{(4)}_*)\leq t.$$
Here, we used \eqref{e.sigma-l-B} for the first equality. Combining this with \eqref{e.decreasing-f-caseIa}, we see that 
$$f(\underline{\theta}^{(1)}\gamma_2\gamma_1b_1\dots b_{j-1}; b_j\dots b_{\widehat{m}}\gamma_2\gamma_1\underline{\theta}^{(2)}) < t - (c_6/2)\cdot |I^u(b_j\dots b_{\widehat{m}}\gamma_2)|.$$
Therefore, in case I.a), we conclude that 
\begin{equation}\label{e.delta1-definition}
f(\sigma^{\ell}(\underline{\theta}^{(1)}\gamma_2;\gamma_1\beta_{\widehat{j}_{p_0}+2}\dots\beta_{\widehat{j}_{p_0}+\widehat{n}-1}\gamma_2\gamma_1\underline{\theta}^{(2)})) < t - \delta_1
\end{equation}
where 
$$\delta_1:=\frac{c_6}{2}\min\limits_{\gamma_1 b_1\dots b_{\widehat{m}}\in\mathcal{B}}\,\,\, \min\limits_{1\leq j\leq \widehat{m}-1}\,\,\, |I^u(b_j\dots b_{\widehat{m}}\gamma_2)| 
> 0$$

The case I.b) is dealt with in a symmetric manner: in fact, by mimicking the argument above for case I.a), one gets that in case I.b) 
\begin{equation}\label{e.delta2-definition}
f(\sigma^{\ell}(\underline{\theta}^{(1)}\gamma_2;\gamma_1\beta_{\widehat{j}_{p_0}+2}\dots\beta_{\widehat{j}_{p_0}+\widehat{n}-1}\gamma_2\gamma_1\underline{\theta}^{(2)})) < t - \delta_2
\end{equation}
where 
$$\delta_2:=\frac{c_6}{2}\min\limits_{\gamma_1 b_1\dots b_{\widehat{m}}\in\mathcal{B}}\,\,\, \min\limits_{1\leq j\leq \widehat{m}-1}\,\,\, |I^s((\gamma_1b_1\dots b_{j-1})^T)| 
> 0$$

Finally, the case II) is dealt with in a similar way to case I.a). We write 
\begin{eqnarray}\label{e.sigma-l-A}
& & \sigma^{\ell}(\underline{\theta}^{(1)}\gamma_2;\gamma_1\beta_{\widehat{j}_{p_0}+2}\dots\beta_{\widehat{j}_{p_0}+\widehat{n}-1}\gamma_2\gamma_1\underline{\theta}^{(2)}) = \\ & & \underline{\theta}^{(1)}\gamma_2 a_1\dots a_{\ell}; a_{\ell+1}\dots a_{\widehat{m}}\beta_{\widehat{j}_{p_0}+2}\dots\beta_{\widehat{j}_{p_0}+\widehat{n}-1}\gamma_2\gamma_1\underline{\theta}^{(2)} \nonumber
\end{eqnarray}
for $0\leq\ell\leq\widehat{m}_1-1$, and 
\begin{eqnarray}\label{e.sigma-l-D}
& & \sigma^{\ell}(\underline{\theta}^{(1)}\gamma_2;\gamma_1\beta_{\widehat{j}_{p_0}+2}\dots\beta_{\widehat{j}_{p_0}+\widehat{n}-1}\gamma_2\gamma_1\underline{\theta}^{(2)}) = \\ & & \underline{\theta}^{(1)}\gamma_2 \gamma_1 \beta_{\widehat{j}_{p_0}+2}\dots\beta_{\widehat{j}_{p_0}+\widehat{n}-1}d_1\dots d_{\ell-\widehat{m_1}-\widehat{m}}; d_{\ell-\widehat{m_1}-\widehat{m}+1}\dots d_{\widehat{m}_2}\gamma_1\underline{\theta}^{(2)} \nonumber
\end{eqnarray}
for $\widehat{m}_1+\widehat{m}\leq\ell\leq \widehat{m}_1+\widehat{m}+\widehat{m}_2-1$.

Since $\widehat{j}_{q_0}-\widehat{j}_{p_0}\geq 2\lceil2/\tau\rceil(q_0-p_0)\geq 2\lceil2/\tau\rceil$ and $\gamma_1,\beta_{\widehat{j}_{p_0}+1},\dots, \beta_{\widehat{j}_{q_0}-1},\gamma_2\in\mathcal{B}_0=\mathcal{C}_{\textbf{u}}(t,r_0)$, we see from the usual bounded distortion property (cf. Remark \ref{r.bounded-distortion}) that, if $r_0$ is sufficiently large, then  
$$|I^u(a_{\ell+1}\dots a_{\widehat{m}}\beta_{\widehat{j}_{p_0}+2}\dots\beta_{\widehat{j}_{p_0}+\widehat{n}-1}\gamma_2)|\leq |I^s((\gamma_2 a_1\dots a_{\ell})^T)|$$
for $0\leq\ell\leq\widehat{m}_1-1$, and 
$$|I^s((\gamma_1 \beta_{\widehat{j}_{p_0}+2}\dots\beta_{\widehat{j}_{p_0}+\widehat{n}-1}d_1\dots d_{\ell-\widehat{m_1}-\widehat{m}})^T)|\leq |I^u(d_{\ell-\widehat{m_1}-\widehat{m}+1}\dots d_{\widehat{m}_2}\gamma_1)|$$
for $\widehat{m}_1+\widehat{m}\leq\ell\leq \widehat{m}_1+\widehat{m}+\widehat{m}_2-1$. By plugging this into the argument for case I.a), one deduces that 
\begin{equation}\label{e.delta3-definition}
f(\sigma^{\ell}(\underline{\theta}^{(1)}\gamma_2;\gamma_1\beta_{\widehat{j}_{p_0}+2}\dots\beta_{\widehat{j}_{p_0}+\widehat{n}-1}\gamma_2\gamma_1\underline{\theta}^{(2)}))<t-\delta_3
\end{equation}
for $0\leq\ell\leq\widehat{m}_1-1$, and 
\begin{equation}\label{e.delta4-definition}
f( \sigma^{\ell}(\underline{\theta}^{(1)}\gamma_2;\gamma_1\beta_{\widehat{j}_{p_0}+2}\dots\beta_{\widehat{j}_{p_0}+\widehat{n}-1} 
\gamma_2\gamma_1\underline{\theta}^{(2)}))<t-\delta_4
\end{equation}
for $\widehat{m}_1+\widehat{m}\leq\ell\leq \widehat{m}_1+\widehat{m}+\widehat{m}_2-1$, where 
$$\delta_3:=\frac{c_6}{2}\min\limits_{\gamma_1 b_1\dots b_{\widehat{m}}\in\mathcal{B}}\,\,\, \min\limits_{1\leq \ell\leq \widehat{m}_1-1}\,\,\, |I^s((\gamma_2 a_1\dots a_{\ell})^T)|
> 0$$
and 
$$\delta_4:=\frac{c_6}{2}\min\limits_{\gamma_1 b_1\dots b_{\widehat{m}}\in\mathcal{B}}\,\,\, \min\limits_{1\leq \ell\leq \widehat{m}_1-1}\,\,\, |I^u(d_{\ell-\widehat{m_1}-\widehat{m}+1}\dots d_{\widehat{m}_2}\gamma_1)|
> 0$$

In summary, from \eqref{e.delta1-definition}, \eqref{e.delta2-definition}, \eqref{e.delta3-definition}, and \eqref{e.delta4-definition}, by setting 
$$\delta:=\min\{\delta_1, \delta_2, \delta_3, \delta_4\}>0,$$
we deduce that \eqref{e.decreasing-f} holds, i.e., 
$$f(\sigma^{\ell}(\underline{\theta}^{(1)}\gamma_2;\gamma_1\beta_{\widehat{j}_{p_0}+2}\dots\beta_{\widehat{j}_{p_0}+\widehat{n}-1}\gamma_2\gamma_1\underline{\theta}^{(2)}))<t-\delta$$
for $0\leq\ell\leq \widehat{m}_1+\widehat{m}+\widehat{m}_2-1$. In other terms, we showed that $\Sigma(\mathcal{B})\subset\Sigma_{t-\delta}$. 

Finally, note that the facts $\Sigma(\mathcal{B}_u)\subset\Sigma_{t-\delta}$ and $\textrm{dim}(K^u(\Sigma(\mathcal{B}_u)))>(1-\eta)D_u(t)$ just established imply that $$D_u(t)\geq d_u(t)\geq d_u(t-\delta)\geq \textrm{dim}(K^u(\Sigma(\mathcal{B})))>(1-\eta) D_u(t).$$  
Since $\eta>0$ is arbitrary, we obtain that $d_u(t)=D_u(t)$. Similarly, by considering $\mathcal{B}_s$, we also have $d_s(t)=D_s(t)$. Furthermore, $\textrm{dim}(K^s(\Sigma(\mathcal{B}_u^T)))>(1-\eta)D_u(t)$ and $\textrm{dim}(K^u(\Sigma(\mathcal{B}_s^T)))>(1-\eta)D_s(t)$ imply that  
$$D_u(t)\geq D_u(t-\delta)\geq \textrm{dim}(K^u(\Sigma(\mathcal{B}_s^T)))>(1-\eta)D_s(t)$$ and 
$$D_s(t)\geq D_s(t-\delta)\geq \textrm{dim}(K^s(\Sigma(\mathcal{B}_u^T)))>(1-\eta)D_u(t),$$
so that $D_u(t)=D_s(t)$. This completes the proof of Proposition \ref{p.tA}.
\end{proof}

\subsection{Approximation of Lagrange spectrum by images of ``subhorseshoes''}

\begin{proposition}\label{p.horseshoe-extraction} Let $f\in\mathcal{R}_{\varphi,\Lambda}$,  $\Sigma(\mathcal{B})\subset\Sigma\subset\mathcal{A}^{\mathbb{Z}}$ be a complete subshift associated to a finite alphabet $\mathcal{B}$ of finite words on $\mathcal{A}$, and $\Lambda(\Sigma(\mathcal{B}))\subset\Lambda$ the subhorseshoe of $\Lambda$ associated to $\mathcal{B}$. 

Then, for each $\varepsilon>0$, there exists a subhorseshoe $\Lambda(\mathcal{B},\varepsilon)\subset \Lambda(\Sigma(\mathcal{B}))$, a rectangle $R(\mathcal{B},\varepsilon)$ of some Markov partition of $\Lambda(\mathcal{B},\varepsilon)$, a $C^1$-diffeomorphism $\mathfrak{A}$ defined in a neighborhood of $R(\mathcal{B},\varepsilon)$ respecting the local stable and unstable foliations, and an integer $j(\mathcal{B},\varepsilon)\in\mathbb{Z}$ such that  $\textrm{dim}(\Lambda(\mathcal{B},\varepsilon))\geq (1-\varepsilon) \textrm{dim}(\Lambda(\Sigma(\mathcal{B})))$ and 
$$f(\varphi^{j(\mathcal{B},\varepsilon)}(\mathfrak{A}(\Lambda(\mathcal{B},\varepsilon)\cap R(\mathcal{B},\varepsilon))))\subset\ell_{\varphi, f}(\Lambda(\Sigma(\mathcal{B})))$$
\end{proposition}

\begin{proof} Note that $f\in\mathcal{R}_{\varphi, \Lambda}$ attains its maximal value on  $\Lambda(\Sigma(\mathcal{B}))$ at finitely many points $x_1, \dots, x_k\in\Lambda(\Sigma(\mathcal{B}))$: this happens because $f\in\mathcal{R}_{\varphi,\Lambda}$ implies that $f$ is locally monotone on stable and unstable manifolds. 

For each $n\in\mathbb{N}$, denote by $\mathcal{B}^n:=\{(\beta_{-n},\dots,\beta_{-1};\beta_0,\dots,\beta_n): \beta_i\in\mathcal{B}, i=1,\dots n\}$.

Given $\varepsilon>0$, we consider the symbolic sequences $(\beta_n^{(j)})_{n\in\mathbb{Z}}\in\mathcal{B}^{\mathbb{Z}}$ associated to $x_j$, $j=1,\dots, k$ and we take\footnote{Compare with Lemma 6 in \cite{MoRo}.} $m=m(\varepsilon, k)\in\mathbb{N}$ large enough such that $m>k^2$ and the subhorseshoe $\Lambda(\mathcal{B},\varepsilon):=\Lambda(\Sigma(\mathcal{B}^*))$ of $\Lambda(\Sigma(\mathcal{B}))$ associated to the complete subshift generated by the alphabet 
$$\mathcal{B}^*:=\mathcal{B}^m -\{\gamma^{(j,m)}:=(\beta_{-m}^{(j)},\dots, \beta_{-1}^{(j)};\beta_0^{(j)},\beta_1^{(j)},\dots,\beta_m^{(j)}): j=1,\dots, k\}$$ 
has Hausdorff dimension 
$$\textrm{dim}(\Lambda(\mathcal{B},\varepsilon))\geq(1-\varepsilon)\textrm{dim}(\Lambda(\Sigma(\mathcal{B}))).$$ 

Denote by $R_a$ the rectangles of the Markov partition of $\Lambda(\Sigma(\mathcal{B}^m))(=\Lambda(\Sigma(\mathcal{B})))$ induced by $\mathcal{B}^m$. By construction, we can find an open set $U\subset M$ such that $\overline{U}$ is disjoint from the Markov rectangles $R_j$ associated to $\gamma_m^{(j)}\in\mathcal{B}^m$, $j=1, \dots, k$, $U\cap\Lambda(\Sigma(\mathcal{B})) = \Lambda(\Sigma(\mathcal{B}))-\bigcup\limits_{j=1}^k R_j$, and 
$\Lambda(\mathcal{B},\varepsilon)= \bigcap\limits_{n\in\mathbb{Z}}\varphi^n(U)$. In particular, there exists $\eta>0$ such that $\max f(\overline{U}) < \min f(\cup_{j=1}^k R_j)-\eta$. Hence, if we take an appropriate periodic point in $\Lambda(\mathcal{B},\varepsilon)$ with a symbolic sequence $\overline{d} = (\dots, d_{\widehat{m}}, d_{-\widehat{m}},\dots,d_{-1};d_0,\dots, d_{\widehat{m}}, d_{-\widehat{m}},\dots)$, then the sequences of the form  
$$\underline{\theta}=(\dots, \beta_{-m-2},\beta_{-m-1},d_{-\widehat{m}},\dots, d_{\widehat{m}}, \gamma^{(1,m)}, d_{-\widehat{m}},\dots, d_{\widehat{m}}, \beta_{m+1},\beta_{m+2}, \dots)$$ 
with $\beta_j\in\mathcal{B}^*$, $|j|>m$, have the property that $m_{\varphi, f}(\underline{\theta})=\sup\limits_{n\in\mathbb{Z}} f(\sigma^n(\underline{\theta}))$ is attained for values of $|n|\leq n_0$ corresponding to the piece $\gamma^{(1,m)}$ and, moreover, $f(\sigma^n(\underline{\theta})) < m_{\varphi, f}(\underline{\theta})-\eta$ whenever $|n|>n_0$ does not correspond to $\gamma^{(1,m)}$. Furthermore, $m_{\varphi, f}(\underline{\theta}') < m_{\varphi, f}(\underline{\theta})-\eta$ for all $\underline{\theta}'\in\Sigma(\mathcal{B}^*)$ and $\underline{\theta}$ as above. 

Next, we denote by $R_1(\mathcal{B}, \varepsilon):=R_{\overline{d}}$ the Markov rectangle associated to the periodic orbit $\overline{d}$ and we define two maps $\theta^*$ and $\underline{\theta}$ from $\Lambda(\mathcal{B},\varepsilon)\cap R_1(\mathcal{B}, \varepsilon)$ to $\Lambda$ in the following way. Given $x\in \Lambda(\mathcal{B},\varepsilon)\cap R_1(\mathcal{B}, \varepsilon)$, let us denote by the corresponding sequence in $\Sigma(\mathcal{B})$ by $(\dots, \gamma_{-1}, d_{-\widehat{m}},\dots; d_0,\dots,  d_{\widehat{m}}, \gamma_1, \dots)$ where $\gamma_j\in\mathcal{B}^*$ for all $j\in\mathbb{Z}$. We set  
$$\tau^{(j)} := (\gamma_{-|j|}, \dots, \gamma_{-1}, d_{-\widehat{m}},\dots, d_{\widehat{m}}, \gamma^{(1,m)}, d_{-\widehat{m}},\dots, d_{\widehat{m}}, \gamma_1, \dots, \gamma_{|j|})$$
for each $j\in\mathbb{Z}$, and we introduce the point $\theta^*(x)$, resp. $\underline{\theta}(x)$, of $\Lambda(\Sigma(\mathcal{B}))$ associated to the sequence 
$$(\dots, \tau^{(-2)}, \tau^{(-1)};\tau^{(1)}, \tau^{(2)},\dots),$$
resp.
$$(\dots, \gamma_{-2},\gamma_{-1},d_{-\widehat{m}},\dots, d_{\widehat{m}}, \gamma^{(1,m)}, d_{-\widehat{m}},\dots, d_{\widehat{m}}, \gamma_{1},\gamma_{2}, \dots)$$

Observe that\footnote{Compare with Lemma 5 in \cite{MoRo}.} the map $\underline{\theta}: \Lambda(\mathcal{B},\varepsilon)\cap R_1(\mathcal{B},\varepsilon)\to \Lambda(\Sigma(B))$ extends to a $C^1$ diffeomorphism $\mathfrak{A}$ on a neighborhood of $R_1(\mathcal{B},\varepsilon)$ respecting the stable and unstable foliations. 

Given $x\in\Lambda(\mathcal{B},\varepsilon)\cap R_1(\mathcal{B},\varepsilon)$, we affirm that 
$$\ell_{\varphi, f}(\theta^*(x)) = f(\sigma^n(\underline{\theta}(x)))$$
for some $|n|\leq n_0$. 

Indeed, we note that $\underline{\theta}(x)$ was constructed in such a way that $m_{\varphi, f}(\underline{\theta}(x)) = f(\sigma^n(\underline{\theta}(x)))$ for some $|n|\leq n_0$. Thus, we have a sequence $(n_k)_{k\in\mathbb{Z}}$ with $n_k$ corresponding to a position in the piece $\tau^{(k)}$ of $\theta^*(x)$ so that $\sigma^{n_k}(\theta^*(x))$ converges to $\sigma^n(\underline{\theta}(x))$ and, \emph{a fortiori}, 
\begin{equation}\label{e.Lagrange-lb}
\ell_{\varphi, f}(\theta^*(x))\geq m_{\varphi, f}(\underline{\theta}(x)) = f(\sigma^n(\underline{\theta}(x))).
\end{equation}
Also, if $(m_k)_{k\in\mathbb{N}}$ and $(r_k)_{k\in\mathbb{N}}$ are sequences where $m_k$ are positions in  pieces $\tau^{(r_k)}$ of $\theta^*(x)$ with the property that 
$$\ell_{\varphi, f}(\theta^*(x)) = \lim\limits_{k\to\infty} f(\sigma^{m_k}(\theta^*(x))),$$ 
then one of the following two possibilities occur: the sequence $|m_k-n_{r_k}|$ has a bounded subsequence or the sequence $|m_k-n_{r_k}|$ is unbounded. In the former case, we can find $b\in\mathbb{Z}$ such that $\sigma^{m_k}(\theta^*(x))$ has a subsequence converging to $\sigma^b(\underline{\theta}(x))$ and, hence, $$\ell_{\varphi, f}(\theta^*(x))=\lim\limits_{k\to\infty} f(\sigma^{m_k}(\theta^*(x)))=f(\sigma^b(\underline{\theta}(x)))\leq m_{\varphi, f}(\underline{\theta}(x)) = f(\sigma^n(\underline{\theta}(x)))$$ 
In the latter case, there exists a subsequence of $\sigma^{m_k}(\theta^*(x))$ converging to an element $\theta'\in \Sigma(\mathcal{B}^*)$, but this is a contradiction with \eqref{e.Lagrange-lb} because 
$$\ell_{\varphi, f}(\theta^*(x)) = \lim\limits_{k\to\infty} f(\sigma^{m_k}(\theta^*(x)))=f(\theta')\leq  m_{\varphi, f}(\theta') < m_{\varphi, f}(\underline{\theta}(x))-\eta$$
This proves our claim. 
    
It follows from our claim that 
$$\Lambda(\mathcal{B}, \varepsilon)\cap R_1(\mathcal{B}, \varepsilon) = 
\bigcup\limits_{j=-n_0}^{n_0}\Lambda(\mathcal{B}, \varepsilon, n)$$
where $\Lambda(\mathcal{B}, \varepsilon, n) = \{x\in\Lambda(\mathcal{B}, \varepsilon)\cap R_1(\mathcal{B}, \varepsilon): \ell_{\varphi, f}(\theta^*(x)) = f(\sigma^n(\underline{\theta}(x)))\}$. Therefore, one of the closed subsets $\Lambda(\mathcal{B}, \varepsilon, j(\mathcal{B}, \varepsilon))$ has non-empty interior in $\Lambda(\mathcal{B}, \varepsilon)$ (for some $|j(\mathcal{B}, \varepsilon)|\leq n_0$) and, \emph{a fortiori}, we can choose a rectangle of $R(\mathcal{B}, \varepsilon)$ of some Markov partition of $\Lambda(\mathcal{B}, \varepsilon)$ such that 
$$\Lambda(\mathcal{B}, \varepsilon)\cap R(\mathcal{B}, \varepsilon)\subset \Lambda(\mathcal{B}, \varepsilon, j(\mathcal{B}, \varepsilon))$$ 
This completes the proof of the proposition. 
\end{proof}

\subsection{Lower semicontinuity of $D_u(t)$ and $D_s(t)$}

Besides Proposition \ref{p.tA} and \ref{p.horseshoe-extraction}, our proof of the lower semicontinuity of $D_u(t)$ uses the dimension formula in \cite{Mo}:

\begin{theorem}\label{t.dimension-formula} There exists a Baire residual subset $\mathcal{U}^{**}\subset \mathcal{U}$ such that for any $g\in\mathcal{R}_{\varphi, \Lambda}$ and for every subhorseshoe $\widetilde{\Lambda}\subset\Lambda$ one has 
$$\textrm{dim}(g(\widetilde{\Lambda})) = \textrm{dim}(\widetilde{\Lambda})$$
\end{theorem}

\begin{remark} The Baire residual subset $\mathcal{U}^{**}$ consists of the diffeomorphisms $\varphi\in\mathcal{U}$ such that all ratios of the logarithms of the multipliers of all pairs of distinct periodic orbits in $\Lambda$ are irrational and all Birkhoff invariants of all periodic orbits in $\Lambda$ are non-zero. Indeed: 
\begin{itemize}
\item by Corollary 1 of \cite{Mo}, the conclusion of Theorem \ref{t.dimension-formula} whenever the ratios of the logarithms of the multipliers of periodic orbits are irrational and a certain property $(H\alpha)$ holds for some $\alpha>0$; 
\item as it is explained in Sections 4 and 9 of \cite{MoYo}, the property $(H\alpha)$ is satisfied whenever all Birkhoff invariants of all periodic points are non-zero. 
\end{itemize}
\end{remark}

At this point, we are ready to conclude the lower semicontinuity of the Hausdorff dimension across generic Lagrange and Markov dynamical spectra:

\begin{proposition}\label{p.Du-lower-sc} There exists a Baire residual subset $\mathcal{U}^{**}\subset \mathcal{U}$ such that, for any $f\in\mathcal{R}_{\varphi, \Lambda}$, the functions $t\mapsto D_u(t)$ and $t\mapsto D_s(t)$ are lower semicontinuous and 
$$D_s(t)+D_u(t) = 2 D_u(t) = \textrm{dim}(L_{\varphi, f}\cap(-\infty, t)) = \textrm{dim}(M_{\varphi, f}\cap (-\infty, t))$$
\end{proposition}

\begin{proof} Let $t\in \mathbb{R}$ with $D_u(t)>0$ and $\eta>0$. By Proposition \ref{p.tA}, we can find $\delta>0$ and a complete subshift $\Sigma(\mathcal{B})\subset\Sigma_{t-\delta}$ on a finite alphabet $\mathcal{B}$ on finite words of $\mathcal{A}$ such that 
$$(1-\eta)(D_s(t)+D_u(t))=2(1-\eta)D_u(t)\leq \textrm{dim}(\Lambda(\Sigma(\mathcal{B})))$$
By Proposition \ref{p.horseshoe-extraction} and Theorem \ref{t.dimension-formula}, we see that 
$$\textrm{dim}(\Lambda(\Sigma(\mathcal{B})))\leq \textrm{dim}(\ell_{\varphi, f}(\Lambda(\Sigma(\mathcal{B})))$$ 

It follows that 
\begin{eqnarray*}
2(1-\eta)D_u(t)&\leq& \textrm{dim}(\Lambda(\Sigma(\mathcal{B})))\leq \textrm{dim}(\ell_{\varphi, f}(\Lambda(\Sigma(\mathcal{B})))) \\ 
&\leq& \textrm{dim}(L_{\varphi, f}\cap(-\infty, t-\delta))\leq \textrm{dim}(M_{\varphi, f}\cap(-\infty, t-\delta)) \\ 
&\leq& \textrm{dim}(f(\Lambda_{t-\delta})) \leq \textrm{dim}(\Lambda_{t-\delta}) \leq 2D_u(t-\delta)
\end{eqnarray*}
Since $\eta>0$ is arbitrary, this proves the proposition.
\end{proof}

\subsection{End of the proof of Theorem \ref{t.A}} Let $\varphi\in\mathcal{U}^{**}$ and $f\in\mathcal{R}_{\varphi, \Lambda}$: note that $\mathcal{U}^{**}$ is residual by Theorem \ref{t.dimension-formula} and $\mathcal{R}_{\varphi,\Lambda}$ is $C^r$-open and dense by Proposition \ref{p.R-generic}. 

By Propositions \ref{p.Du-upper-sc} and \ref{p.Du-lower-sc}, the function 
$$t\mapsto D_s(t)=D_u(t)=\frac{1}{2}\textrm{dim}(L_{\varphi, f}\cap(-\infty, t))=\frac{1}{2}\textrm{dim}(M_{\varphi, f}\cap (-\infty, t))$$
is continuous. 

Since Proposition \ref{p.tA} says that $d_s(t)=D_s(t)=D_u(t)=d_u(t)$ for all $t\in\mathbb{R}$, the proof of Theorem \ref{t.A} is complete.


\begin{thebibliography}{SCh95b}

\bibitem[CF]{CF}
T.~Cusick and M.~Flahive, 
\newblock{The Markoff and Lagrange spectra}, 
\emph{Mathematical Surveys and Monographs}, \textbf{30}. American Mathematical Society, Providence, RI, 1989. x+97 pp. 

\bibitem[HP]{HP}
S.~Hersonsky and F.~Paulin, 
\newblock{Diophantine approximation for negatively curved manifolds},
\emph{Math. Z.} \textbf{241} (2002), no. 1, 181-226. 

\bibitem[HPS]{HPS}
M.~Hirsch, C~Pugh and M.~Shub, 
\newblock{Invariant manifolds.}
\emph{Lecture Notes in Mathematics}, Vol. \textbf{583}. Springer-Verlag, Berlin-New York, 1977. ii+149 pp.

\bibitem[Mo1]{Mo1}
C.~G.~Moreira,
\newblock{Geometric properties of the Markov and Lagrange spectra}, Preprint (2016) available at arXiv:1612.05782, accepted for publication on Ann. Math. 

\bibitem[Mo2]{Mo}
C.~G.~Moreira,
\newblock{Geometric properties of images of cartesian products of regular Cantor sets by differentiable real maps}, Preprint (2016) available at arXiv:1611.00933

\bibitem[MoRo]{MoRo}
C.~G.~Moreira and S.~Roma\~na,
\newblock{On the Lagrange and Markov dynamical spectra}, \emph{Ergodic Theory Dynam. Systems} 37 (2017), 1570--1591. 

\bibitem[MoYo]{MoYo}
C.~G.~Moreira and J.-C.~Yoccoz,
\newblock{Tangences homoclines stables pour des ensembles hyperboliques de grande dimension fractale}, Ann. Sci. \'Ec. Norm. Sup\'er. (4) 43 (2010), 1--68.  

\bibitem[PT]{PT}
J.~Palis and F. Takens, 
\newblock{Hyperbolicity and sensitive chaotic dynamics at homoclinic bifurcations.}
\emph{Cambridge Studies in Advanced Mathematics}, \textbf{35}. Cambridge University Press, Cambridge, 1993. x+234 pp. ISBN: 0-521-39064-8. 

\end{thebibliography}
\end{document}